\documentclass{amsart}
\usepackage{amsmath,amssymb,amsxtra,graphicx,tikz,tikz-cd,anysize,enumerate,hyperref} 
\usepackage{amsmath}
\usepackage{tikz}
\usepackage{tikz-cd}
\usepackage{verbatim}
\usetikzlibrary{arrows.meta,decorations.pathreplacing}
\usepackage{tikz}
\usetikzlibrary{arrows.meta}
\usepackage{subcaption} 

\usepackage{array}
\usetikzlibrary{calc}%
\usetikzlibrary{shapes}%
\usetikzlibrary{patterns}%
\usetikzlibrary{positioning}%
\usetikzlibrary{arrows.meta}
\usetikzlibrary{knots}
\usetikzlibrary{decorations.markings}
\DeclareMathOperator{\Conv}{Conv}

\newcommand{\Z}{\mathbb{Z}}
\newcommand{\C}{\mathbb{C}}
\newcommand{\R}{\mathbb{R}}

\newcommand{\Fl}{\mathrm{Fl}}
\newcommand{\Sl}{\mathrm{SL}}
\newcommand{\Gr}{\mathrm{Gr}}
\newcommand{\Spec}{\mathrm{Spec}}
\newcommand{\Mat}{\mathrm{Mat}}
\newcommand{\MaRi}{\mathrm{MR}}

\newcommand{\PP}{\mathbb{P}}
\newcommand{\AAA}{\mathbb{A}}
\newcommand{\sgn}{\mathrm{sgn}}
\newcommand{\TT}{\mathbb{T}}

\newcommand{\sR}{{\tt R}}
\newcommand{\sm}{{\tt m}}

\newcommand{\CF}{\mathcal{F}}
\newcommand{\CM}{\mathcal{M}}

\newcommand{\A}{\mathrm{A}}
\newcommand{\B}{\mathrm{B}}
\newcommand{\CC}{\mathrm{C}}

\def\bw{\mathbf{w}}
\DeclareMathOperator{\wt}{wt}

\numberwithin{equation}{section}

\newtheorem{theorem}{Theorem}[section]
\newtheorem{proposition}[theorem]{Proposition}
\newtheorem{corollary}[theorem]{Corollary}
\newtheorem{lemma}[theorem]{Lemma}
\newtheorem{question}[theorem]{Question}
\theoremstyle{definition}
\newtheorem{remark}[theorem]{Remark}

\newtheorem{definition}[theorem]{Definition}
\newtheorem{example}[theorem]{Example}

\DeclareMathOperator{\Int}{Int}

\newcommand{\SK}[1]{{\color{orange}{[Soyeon : #1]}}}

\newcommand{\EG}[1]{{\color{red}{[EG: #1]}}}

\title{Unexpected toric Richardson varieties}
\date{}

\author{Eugene Gorsky}
\address{Department of Mathematics, University of California Davis\\ One Shield Avenue, Davis CA 95616 USA}
\email{egorskiy@ucdavis.edu}
\author{Soyeon Kim}
\address{Department of Mathematics, University of California Davis\\ One Shield Avenue, Davis CA 95616 USA}
\email{syxkim@ucdavis.edu}
\author{Melissa Sherman-Bennett}
\address{Department of Mathematics, University of California Davis\\ One Shield Avenue, Davis CA 95616 USA}
\email{mshermanbennett@ucdavis.edu}

\begin{document}

\begin{abstract}
We prove that an open Richardson variety in the complete flag variety for $\mathrm{GL}_n$ is isomorphic to a torus if and only if the corresponding closed Richardson variety is toric. 
Such toric varieties can be classified in terms of the combinatorics of Bruhat intervals, and include many varieties of dimension larger than $n-1$. 
We give a combinatorial description of the corresponding polytopes, and compute several explicit examples.
\end{abstract}

\maketitle



\section{Introduction}

Richardson varieties arise in various contexts in algebraic geometry and combinatorics, including Schubert calculus, total positivity, and cluster algebras. Given two permutations $v,w\in S_n$ such that $v\le w$ in Bruhat order, the \emph{open Richardson variety} $R_{v,w}$ is the intersection of the opposite Schubert cell for $v$ and the Schubert cell for $w$. Open Richardson varieties are smooth, affine, irreducible varieties, and $\dim R_{v,w} = \ell(w)-\ell(v)$. The \emph{closed Richardson variety} $\overline{R}_{v,w}$ is the Zariski closure of $R_{v,w}$ in the complete flag variety $\Fl(n)$.  This paper is concerned with the closed Richardson varieties which are toric varieties. Recall that $\overline{R}_{v,w}$ is $\TT$-toric for some torus $\TT$ if there is a $\TT$-action on $\overline{R}_{v,w}$ with an open orbit. We say $\overline{R}_{v,w}$ is toric if it is $\TT$-toric for some $\TT$.

The $(n-1)$-dimensional torus $T\subset \Sl(n)$ of diagonal matrices with determinant 1 acts on the flag variety and on all Richardson varieties. 
It is well-known (see \cite{A19,Can,Escobar,GH, LMP,TW15} and references therein) that $\overline{R}_{v,w}$ is $T$-toric if and only if the roots associated to any maximal chain from $v$ to $w$ are linearly independent. By construction, all $T$-toric Richardson varieties have dimension at most $n-1$. 

It is also possible for a Richardson variety $\overline{R}_{v,w}$ to be toric with respect to some other torus $\TT$. 
This situation has been much less studied. As shown in \cite[Corollary 3.3]{Can}, if $\overline{R}_{v,w}$ is $\TT$-toric, then the interval $[v,w]$ is a lattice, or, equivalently, has no subintervals which are 2-crowns (a 2-crown is the Bruhat graph of $S_3$ shown below).
\[\begin{tikzcd}
	& \bullet \\
	\bullet && \bullet \\
	\bullet && \bullet \\
	& \bullet
	\arrow[no head, from=1-2, to=2-3]
	\arrow[no head, from=2-1, to=1-2]
	\arrow[no head, from=2-1, to=3-1]
	\arrow[no head, from=2-1, to=3-3]
	\arrow[no head, from=2-3, to=3-1]
	\arrow[no head, from=2-3, to=3-3]
	\arrow[no head, from=3-1, to=4-2]
	\arrow[no head, from=4-2, to=3-3]
\end{tikzcd}\]

The first main result of this paper is that, if the open Richardson variety $R_{v,w}$ is a torus $\TT$, then $\overline{R}_{v,w}$ is a toric variety.

\begin{theorem}
\label{thm: main intro}
Suppose the open Richardson variety $R_{v,w}$ is isomorphic to a torus $\TT$. Then the action of $\TT$ on $R_{v,w}$ extends to the closed Richardson variety $\overline{R}_{v,w}$ and $\overline{R}_{v,w}$ is a toric variety with respect to $\TT$.
\end{theorem}

We then show that this is in fact a characterization of toric Richardsons, and also provide various equivalent combinatorial characterizations.

\begin{theorem}
\label{thm: intro classification}
The following conditions are equivalent:
\begin{enumerate}[(a)]
\item The open Richardson variety $R_{v,w}$ is a torus.
\item The closed Richardson variety $\overline{R}_{v,w}$ is a toric variety. 
\item The Bruhat interval $[v,w]$ does not contain a subinterval isomorphic to $S_3$.
\item The Bruhat interval $[v,w]$ is a lattice.
\item The interval poset $\{[v', w']: [v', w'] \subset [v,w]\} \cup \{\emptyset\}$, ordered by containment, is a lattice.
\end{enumerate}
\end{theorem}
We prove (a) $\Leftrightarrow$ (b) in Theorem~\ref{thm: toric} and (a) $\Leftrightarrow$ (c) in Lemma~\ref{lem: 2 crown}. The equivalence (d) $\Leftrightarrow$ (e) is an easy exercise in poset theory; we include a proof in Lemma~\ref{lem:poset-lattice-interval} for completeness. The implication (d) $\implies$ (c) is straightforward, as $S_3$ is not a lattice. The reverse implication (c) $\implies$ (d) is an unpublished result of Dyer \cite{Dyer}; alternately, one can argue (c) $\implies$ (b) $\implies$ (e) $\implies$ (d), where the middle implication follows from Theorem~\ref{thm: polytope intro} below.


We call toric Richardson varieties 
{\bf unexpected} if $\dim R_{v,w}=d>n-1$, since in this case they are \emph{not} $T$-toric. Unexpected toric Richardsons are toric with respect to a torus action that does not extend to the entire flag variety. We give a number of explicit examples of unexpected toric Richardsons in Section~\ref{sec:examples}. One nice infinite family is given by the ``large hypercube" intervals recently found in \cite{hypercube}; in this case, $d = O(n \log n)$, which shows the dimensions of unexpected toric Richardsons can grow quite large. As explained in \cite[Section 3.4]{hypercube}, under the assumptions of Theorem \ref{thm: intro classification} the closed Richardson variety $\overline{R}_{v,w}$ is smooth if and only if $[v,w]$ is a hypercube.

One way to understand a toric Richardson $\overline{R}_{v,w}$ is to study its moment polytope $\widehat{P}(v,w)$. For $T$-toric Richardsons, the moment polytope is the \emph{Bruhat interval polytope} $P(v,w)$ of \cite{KodamaWilliams}, see \cite{TW15}. For $T$-toric Richardsons, the face lattice of $P(v,w)$ is the interval poset $\Int[v,w]$, which consists of the subintervals of $[v,w]$ and $\emptyset$, ordered by containment. In particular, the Hasse diagram of $[v,w]$ is the 1-skeleton of $P(v,w)$. 
 A similar result holds for arbitrary toric Richardsons.

\begin{theorem}
\label{thm: polytope intro}
The moment polytope $\widehat{P}(v,w)$ of the toric Richardson $\overline{R}_{v,w}$ has face lattice isomorphic to $\Int[v,w]$. That is, the $s$-dimensional faces of $\widehat{P}(v,w)$ are in bijection with subintervals $[v',w']\subset [v,w]$ with $\ell(w')-\ell(v')=s$, and this bijection respects containment.

In particular, the 1-skeleton of $\widehat{P}(v,w)$ is the Hasse diagram of $[v,w]$.
\end{theorem}

One nice consequence of our results is a characterization of when $\Int[v,w]$ is the face lattice of a polytope. We note that if $[v,w]$ is a 2-crown, then $\Int[v,w]$ is not a lattice so cannot be the face lattice of a polytope. In fact, 2-crowns are the only obstacle for $\Int[v,w]$ being a face lattice.

\begin{corollary} The following are equivalent:
\begin{enumerate}[(a)]
    \item The interval poset $\Int[v,w]$ of sub-intervals of $[v,w]$ is the face lattice of a polytope.
    \item The Bruhat interval $[v,w]$ does not have any 2-crowns.
    \item The Bruhat interval $[v,w]$ is a lattice.
\end{enumerate}
\end{corollary}

Another way to understand toric Richardson varieties is to study projections to Grassmannians. By \cite{KLS09,KLS13}, the projection of $\overline{R}_{v,w}$ to $\Gr(k,n)$ is a \emph{positroid variety}, which is indexed by, among other things, a positroid and a plabic graph. We show that all projections of toric $\overline{R}_{v,w}$ are $T$-toric positroid varieties, whose corresponding plabic graphs $G_k(v,w)$ are forests. We can also understand the moment polytope $\widehat{P}(v,w)$ using \emph{positroid polytopes}, which are $T$-moment polytopes of positroid varieties. 

\begin{theorem}
\label{thm: minkowski sum intro}
Assume that $\overline{R}_{v,w}$ is toric. 
The moment polytope $\widehat{P}(v,w)$ is the Minkowski sum of affine transformations of positroid polytopes indexed by $G_k(v,w)$.
\end{theorem}


Theorem \ref{thm: minkowski sum intro} generalizes the fact that Bruhat interval polytopes are Minkowski sums of {positroid polytopes}. 
Note that the affine transformations in Theorem \ref{thm: minkowski sum intro} in general depend on $k$, and map the positroid polytopes contained in $\R^{n-1}$ to the  space $\R^d$ corresponding to $\TT$.




Though we do not pursue this direction here, the motivation for this paper arose from torus actions from cluster structures. In recent years, it was proved \cite{CGGLSS,GLSBS1,GLSBS2} that the open Richardson variety $R_{v,w}$ has a cluster structure. As a consequence, one can construct \cite{LS,Soyeon} an action of a potentially larger torus $\TT$ on $R_{v,w}$, where the dimension of $\TT$ equals the number of {\em frozen variables} in the cluster structure\footnote{The number of frozen variables is a combinatorial invariant of the Bruhat interval \cite{Patimo} and was recently shown to coincide with the dimension of $H^1(R_{v,w})$ and the dimension of an $\text{Ext}$ group between two Verma modules, see \cite{BGL26}.}. This leads to the following question.

\begin{question}
\label{question: intro}
Does the action of $\TT$ on $R_{v,w}$ extend to the closed Richardson variety $\overline{R}_{v,w}$?
\end{question}

Theorem \ref{thm: main intro} answers Question \ref{question: intro} in the positive in the case when $\TT$ has maximal dimension. \\

The structure of the paper is as follows. In Section~\ref{sec:background} we give background on Richardson varieties, positroid varieties, and related combinatorics. In Section~\ref{sec:characterization}, we characterize the closed Richardson varieties which are toric and investigate their projections to the Grassmannian. In Section~\ref{subsec: polytopes} we discuss the moment polytopes of toric Richardson varieties. Finally, Section~\ref{sec:examples} has a number of examples, including two infinite families of unexpected toric Richardson varieties.

\section*{Acknowledgments}

The authors would like to thank Jos\'e Simental for his constant support and interest in this work, and for explaining Proposition \ref{prop: hypercube matroid} to us. We would also like to thank David Anderson, Sara Billey, Mahir Bilen Can, Laura Escobar, Christian Gaetz, Mikhail Gorsky and Geordie Williamson for useful discussions. E. G. and S. K. were partially supported by the NSF grant DMS-2302305. M.S.B. was partially supported by NSF grant DMS-2444020.

\section{Background}\label{sec:background}

We use the notation $[n]:=\{1, \dots, n\}$, $[i,j]:=\{i, i+1, \dots, j\}$,   $\binom{[n]}{k}:= \{I \subset [n]: |I|=k\}$, and, for $u \in S_n$, $u[k]:=\{u(1), \dots, u(k)\}$. More generally, for $I \subset [n]$ and $u \in S_n$, $u(I):=\{u(i): i \in I\}$. Permutations are written in one-line notation with the occasional exception of transpositions, which we may write as $(i~j)$ in cycle notation. We use $s_i$ to denote the simple transposition $(i~i+1)$.

\subsection{Algebraic tori}

A torus is an algebraic variety isomorphic to $(\C^\times)^d$ for some $d$. A specific choice of isomorphism $\TT\simeq (\C^\times)^d$ defines coordinates on $\TT$, and coordinate-wise multiplication equips $\TT$ with a structure of an algebraic group, which we will refer to below as {\em multiplicative structure}. Note that different choices of coordinates might give different multiplicative structures.  A multiplicative structure gives an isomorphism $\C[\TT] \cong \C[x_i^{\pm1}]$, and the only units in this ring are Laurent monomials in the $x_i$.

We will repeatedly use the following simple lemma.

\begin{lemma}
\label{lem: open torus}
Suppose $\TT$ is a torus, and $U$ is an open subset in $\TT$. If $U$ is a torus then $U=\TT$.
\end{lemma}

\begin{proof}
We have $\dim U=\dim \TT=d$. Let us choose some coordinates $x_1,\ldots,x_d$ on $U$ and $y_1,\ldots,y_d$ on $\TT$. The restrictions of $y_i$ on $U$ are invertible hence they are Laurent monomials in $x_i$:
\begin{equation}
\label{eq: restrict torus}
y_i|_{U}=c_i\prod x_j^{a_{ij}},\ a_{ij}\in \Z,\ c_i\neq 0.
\end{equation}
Define $y'_i=y_i/c_i$ and consider the multiplicative structures on $\TT$ (resp. $U$) defined by $y'_i$ (resp. $x_i$). By \eqref{eq: restrict torus} the inclusion $U\hookrightarrow \TT$ is a homomorphism with respect to these multiplicative structures, so $U$ is a subgroup of $\TT$. Since $\dim U=\dim \TT$, we get $U=\TT$.
\end{proof}

We will need the following easy fact about cluster varieties. We refer to \cite{Muller} and references therein for more details and definitions.  

\begin{corollary}
\label{cor: all frozen}
Suppose $\mathcal{A}$ is a cluster algebra. Then $X=\Spec \mathcal{A}$ is a torus if and only if all cluster variables in $\mathcal{A}$ are frozen.    
\end{corollary}

\begin{proof}
Let $\mathfrak{s}$ be one seed for $\mathcal{A}$. By \cite[Section 2.6]{Muller}, it corresponds to an open torus $T_{\mathfrak{s}}\subset X$. If all cluster variables in $\mathcal{A}$ are frozen then $\mathcal{A}$ is the Laurent polynomial algebra in the frozen variables so $T_{\mathfrak{s}}=X$ by definition.

Conversely, assume  $X$ is a torus. Then by Lemma \ref{lem: open torus} we get $X=T_{\mathfrak{s}}$, so all cluster variables are invertible in $\mathcal{A}$ and hence frozen by \cite[Theorem 1.3]{GLS}.
\end{proof}

\subsection{Richardson varieties}

We refer to \cite{SpeyerNotes} for all background on Richardson varieties and positroid varieties. We fix the Borel subgroup $B \subset \Sl(n)$ of upper triangular matrices. 
We consider the Schubert cell $X^w = BwB/B$ in the complete flag variety $\Fl(n) \cong \Sl(n)/B$, and the opposite Schubert cell $X_v=B_- vB/B$. These are isomorphic to affine spaces of dimensions $\ell(w)$ and $\binom{n}{2}-\ell(v)$ respectively.

\begin{definition}
The open Richardson variety is defined as the intersection $R_{v,w}=X^w\cap X_v$. The closed Richardson variety $\overline{R}_{v,w}$ is the Zariski closure of $R_{v,w}$ inside $\Fl(n)$.
\end{definition}

The open Richardson variety $R_{v,w}$ is nonempty if and only if $v \leq w$ in the Bruhat order. The dimension of both $R_{v,w}$ and $\overline{R}_{v,w}$ equals $\ell(w)-\ell(v)$. The variety $R_{v,w}$ is smooth, affine, and irreducible, and is a cluster variety \cite{CGGLSS,GLSBS1}. The variety $\overline{R}_{v,w}$ is typically singular \cite{KWY,Billey}.
We have decompositions 
\[\overline{R}_{v,w} = \bigsqcup_{[v', w'] \subset [v,w]} R_{v',w'} \qquad \text{and} \qquad \overline{R}_{v,w} = \bigcup_{[v', w'] \subset [v,w]} \overline{R}_{v',w'}.  \]

We define the {\em $\sR$-polynomial} $\sR_{v,w}(q)$ as the point count in $R_{v,w}$ over the finite field with $q$ elements.

Given an $n\times n$ nondegenerate matrix $M$ with columns $M_i$, one can define the flag
$$
\CF_{M}=\left\{
\langle M_1\rangle \subset \langle M_1,M_2\rangle\subset \cdots
\subset \langle M_1,\ldots,M_n\rangle=\C^n\right\}.
$$
Given $I \in \binom{[n]}{k}$, we define the {\em flag minor} (or {\em Pl\"ucker coordinate}) $\Delta_I=\Delta_I(M)$ as the $k\times k$-minor of the matrix $M$ using rows $I$
and the first $k$ columns. Choosing a different representative matrix $M'$ for the same flag yields proportional minors $\Delta_{I}(M')=c_k\Delta_{I}(M)$ where the constant $c_k\neq 0$ does not depend on $I$ but might depend on $k$. 

Alternatively, one can think of the projection $\pi_k:\Fl(n)\to \Gr(k,n)$ which sends the flag $\CF_{M}$ to the $k$-dimensional subspace $\langle M_1,\ldots,M_k\rangle$. Then the minors $\Delta_I(M)$ with $I \in \binom{[n]}{k}$ are the Pl\"ucker coordinates of $\pi_k(\CF_M)$. 

We will need to know which  minors are identically zero on $\overline{R}_{v,w}$ and which are not, which is essentially \cite[Theorem 7.1]{TW15}. Note that $\Delta_I$ is identically zero on $\overline{R}_{v,w}$ if and only if it is identically zero on $R_{v,w}$.

\begin{definition}\label{def: k constituent}
We define $\CM_k(v,w)$ as the collection of all $k$-subsets $I\subset [n]$ such that the minor $\Delta_I$ is not identically zero on $R_{v,w}$. We call $\CM_k(v,w)$ the \emph{$k$th constituent} of the interval $[v,w]$.
\end{definition}

Note that many different intervals can have the same $k$th constituent, but the following result shows that the collection of all constituents uniquely determines an interval.

\begin{theorem}
\label{thm: flag minor richardson}
  Fix $I\in \binom{[n]}{k}$. The flag minor $\Delta_{I}$ is generically nonzero on the closed Richardson variety $\overline{R}_{v,w}$ if and only if there is some $u \in [v,w]$ such that $I=u[k]$. In other words,
$$
\CM_k(v,w)=\left\{u[k]\ :\ u\in [v,w]\right\}.
$$
\end{theorem}

\begin{proof} For a tuple of nested subsets $\mathbf{I}=(I_1 \subsetneq I_2 \subsetneq \dots \subsetneq I_{n-1})$, define the \emph{generalized Pl\"ucker coordinate} as the product of flag minors $\Delta_{\mathbf{I}}:= \Delta_{I_1} \dots \Delta_{I_{n-1}}$. Then
   \cite[Theorem 7.1]{TW15} shows that for any point $F$ in the positive part $R_{v,w}^{>0}$, the generalized Plucker coordinate $\Delta_{\mathbf{I}}$ is nonzero if and only if $\mathbf{I}=(u[1] \subset u[2] \subset \cdots \subset u[n-1])$ for some $u \in [v,w]$. 

    On the other hand, the nonzero flag minors of $F$ define a flag matroid. By the Maximality Property for flag matroids \cite{CoxeterMatroid},
    any nonzero flag minor $\Delta_J$ of $F$ is a factor of some nonzero generalized Pl\"ucker coordinate. This implies that for $F \in R_{v,w}^{>0}$, the flag minor $\Delta_J$ is nonzero if and only if $J=u[k]$ for some $u \in [v,w]$.

    Since the positive part $R_{v,w}^{>0}$ is Zariski dense in the closed Richardson, this implies the desired statement.
\end{proof}

Finally, we mention some compatibility relations between the minors of different sizes.
The minors $\Delta_I$ satisfy incidence Pl\"ucker relations which are defined as follows.
For $1\le r\le s\le n$ and for each $I \in \binom{[n]}{r-1}, J \in \binom{[n]}{s+1}$, we have the relation
\begin{equation}
\label{eq: incidence Plucker}
E_{I,J}=\sum_{j\in J\setminus I}\sgn(j,I,J) \Delta_{I\cup j}\Delta_{J\setminus j}=0
\end{equation}
where the signs are defined by $\sgn(j,I,J)=(-1)^{|\{k\in J:k<j\}|+|\{i\in I:i>j\}|}$. In the case $r=s$ the relations \eqref{eq: incidence Plucker} are the usual Pl\"ucker relations in the Grassmannian $\Gr(r,n)$.  

\subsection{Positroids and positroid varieties}
The collection $\CM_k(v,w)$ of Definition~\ref{def: k constituent} is an example of a matroid known as a \emph{positroid}, and all positroids can be obtained this way by varying $v$ and $w$. Moreover, we have $\CM_k(v,w)=\CM_k(v',w')$ for a unique $v'<w'$ where $w'$ is $k$-Grassmannian, meaning that $w'(1)<w'(2) <\dots< w'(k)$ and $w'(k+1) < \dots <w'(n)$.
We will not need the precise definition of a positroid and refer to \cite{KLS09,KLS13,Postnikov} for more details and definitions. However, we will need three constructions related to positroids.

First, given a positroid $\CM=\CM_k(v,w)$ where $w$ is $k$-Grassmannian, one can define the {\emph {open positroid variety}} $\Pi_{\CM}^{\circ}\subset \Gr(k,n)$ as the projection of the open Richardson variety
$$
\Pi_{\CM}^{\circ}=\pi_k(R_{v,w}).
$$
This projection is an isomorphism. In particular, we have
\begin{equation}
\label{eq: def open positroid}
\Pi_{\CM}^{\circ} \supset\{V\in \Gr(k,n)\ :\ \Delta_{I}(V)\neq 0\ \mathrm{if\ and\ only\ if}\ I\in \CM\}.
\end{equation}
An open positroid variety is a cluster variety by \cite{GLpositroid}. We denote by $\Pi_{\CM}$ the closure of $\Pi_{\CM}^{\circ}$. We have the decomposition
\[\Pi_{\CM} = \bigsqcup_{\CM' \subset \CM} \Pi_{\CM'}^\circ.\]

Second, it is convenient to encode positroids by {\emph {plabic graphs}} following \cite{Postnikov}. A plabic graph is a planar, bicolored graph in a disk with $n$ leaves on the boundary of the disk\footnote{Note that the boundary of the disk is not part of the plabic graph.} labeled by $1,\ldots,n$. A plabic graph is called reduced if it satisfies some technical condition in \cite[Definition 12.5]{Postnikov}.  
Note that any plabic forest (plabic graph with no cycles) is automatically reduced. 

\begin{definition}
Given a plabic graph $G$, an orientation of $G$ is \emph{perfect} if each white vertex $v$ has exactly one edge directed towards $v$, and each black vertex $v$ has exactly one edge directed away from $v$.
\end{definition}

Given a perfect orientation $\mathcal{O}$ on a reduced plabic graph $G$,  we can consider the subset $I_{\mathcal{O}} \subset [n]$ of boundary vertices which are sources. By \cite[Proposition 11.7]{Postnikov} the collection of subsets $I_{\mathcal{O}}$ for all possible perfect orientations $\mathcal{O}$ forms a positroid which we will denote by $\CM_{G}$. Furthermore, this construction is a  bijection between (equivalence classes of) reduced plabic graphs and positroids.

Third, we need the notion of positroid polytopes. For $I \subset [n]$, let $e_I :=\sum_{i \in I} e_i \in \R^n$ be the indicator vector of $I$.
\begin{definition}\label{def: positroid polytope}
    For a positroid $\CM= \CM_k(v,w)$, we define the \emph{positroid polytope}
    \[P_k(v,w):= \Conv(e_I: I \in \CM) \subset \R^n.\]
\end{definition}
The positroid polytope $P_k(v,w)$ is the image of the closed positroid variety $\Pi_k(v,w)$ under the moment map for the standard torus $T$.

Finally we have a lemma which we will make use of later.

\begin{lemma}
\label{lem: forest} 
Suppose $\CM$ is a positroid such that the open positroid variety $\Pi_{\CM}^{\circ}$ is a torus. Then the corresponding plabic graph $G_{\CM}$ is a forest.
\end{lemma}

\begin{proof}
By \cite{GLpositroid} the cluster variables for a seed for $\Pi_{\CM}^{\circ}$ correspond to the faces of the plabic graph $G_{\CM}$. Faces that are not adjacent to the boundary of the disk, which are called internal faces, correspond to mutable variables. Corollary \ref{cor: all frozen} then implies that there are no internal faces. Any plabic graph $G$ which has a cycle has an internal face, so this implies $G_{\CM}$ is a forest.
\end{proof}


\subsection{Deodhar decomposition and Marsh-Rietsch parametrization} 
Choose a reduced expression $\bw$ for $w \in S_n$. The Richardson variety $R_{v,w}$ admits a Deodhar decomposition \cite{Deodhar} 
\[\bigsqcup_{\sigma} (\C^{\times})^{a(\sigma)}\times \C^{b(\sigma)} \]
where the components are indexed by \emph{distinguished subexpressions} $\sigma$ for $v$ in $\bw$. We have that $a(\sigma)+2b(\sigma)=\ell(w)-\ell(v)=d$. There is a unique component of top dimension, which is open in $R_{v,w}$ and is also the unique component which is a torus. We call this component the \emph{Deodhar torus}. 

\begin{lemma}\label{lem:torus iff one deodhar comp} Fix a reduced word $\bw$ for $w \in S_n$.
    The open Richardson $R_{v,w}$ is isomorphic to a torus if and only if the Deodhar torus is the unique Deodhar component.
\end{lemma}
\begin{proof}
    The forward direction follows from Lemma~\ref{lem: open torus}, and the backwards direction from the fact that $R_{v,w}$ is the union of the Deodhar components.
\end{proof}

Marsh--Rietsch \cite{MR} give a parametrization of all Deodhar components. Here, we recall their results for the Deodhar torus.

Let $\phi_i : \Sl(2) \hookrightarrow \Sl(n)$ send $M \in \Sl(2)$ to the $n \times n$ matrix which has $M$ in the $2 \times 2$ block in rows and columns $i, i+1$, and otherwise agrees with the identity matrix. We define
\[\dot{s}_i := \phi_i\left(\begin{bmatrix}
    0& -1\\
    1 & 0
\end{bmatrix}\right) \qquad \text{and} \qquad y_i(t):= \phi_i\left( \begin{bmatrix}
    1& 0\\
    t & 1
\end{bmatrix} \right)\]

\begin{definition}[MR parametrization]
    Let $\bw=s_{i_1} \dots s_{i_\ell}$ be a reduced expression for $w$. By \cite[Lemma 3.4]{MR}, $\bw$ has a unique ``rightmost" reduced subexpression for $v$, called the \emph{positive distinguished subexpression} (PDS). Let $J^+ \subset [\ell]$ denote the indices used in the PDS and $J^{\circ}:= [\ell] \setminus J^+$. We define
    \[g_j:= \begin{cases}
        \dot{s}_{i_j} & \text{if}~j \in J^+\\
        y_{i_j}(t_j) & \text{if}~ j \in J^{\circ}
    \end{cases}\]
    and set $g:=g_1 g_2 \cdots g_{\ell}$. We call $\{t_j : j \in J^\circ\}$ the \emph{MR parameters} and $g$ the \emph{MR matrix} for $(v, \bw)$.
\end{definition}
To ease notation, we omit the dependence of $g$ on the MR parameters and the reduced word $\bw$.

\begin{example}\label{ex:running ex for 3d}
    Let $v=1324$ and $w= 4231$. We choose reduced expression $\bw=s_1s_2s_3\underline{s_2}s_1$ for $w$. The PDS for $v$ in $\bw$ is underlined, so $J^+=\{4\}$ and $J^\circ=\{1,2,3,5\}$. The associated MR matrix is \begin{equation*}\label{eq:braid matrix for w}
    g= y_1(t_1)~y_2(t_2)~ y_3(t_3)~ \dot{s}_2~ y_1(t_5)=\left(
\begin{matrix}
1 & 0 & 0 & 0 \\ t_1 & 0 & -1 & 0 \\ t_5 & 1 & -t_2 & 0 \\ t_3t_5 & t_3 & 0 & 1
\end{matrix}
\right).
\end{equation*}
\end{example}

\begin{proposition}[{\cite[Proposition 5.2]{MR}}] \label{prop:MR-param}
   The map 
\begin{align*}
    (\C^\times)^{J^\circ} &\to R_{v,w}\\
    (t_j)_{j \in J^\circ} &\mapsto g B
\end{align*}
is an injection whose image is the Deodhar torus in $R_{v,w}$. In particular, if $R_{v,w}$ is a torus, this map is an isomorphism.
\end{proposition}

We note that there is a description of $g$ as using the (directed, weighted) \emph{3D plabic graph} $G_{v, \bw}$, defined as follows. We associate to each factor $g_j$ one of the two ``chips" below. We use the one on the left, with a vertical edge called a \emph{bridge}, if $g_j= y_{i_j}(t_j)$ (equivalently, $j \in J^\circ$) and the one on the right, called a \emph{crossing}, if $g_j=\dot{s}_{i_j}$ (equivalently if $j \in J^+$).



\begin{center}
    \begin{tikzpicture}[scale=1.2, line width=0.8pt]

\begin{scope}[shift={(0,0)}]
    \foreach \y/\label in {0/1, 0.5/2, 1.0/i_j, 1.5/i_{j+1}, 2.0/n} {
        \draw [-{Stealth[scale=0.6]}] (0,\y) -- (1,\y);
        \draw (1,\y) -- (2.5,\y);
        \node[left] at (0,\y) {$\label$};
    }
    
    \draw[thick, dotted] (1.25, 0.65) -- (1.25, 0.85);
    \draw[thick, dotted] (1.25, 1.65) -- (1.25, 1.85);
    
    \draw[-{Stealth}] (1.25, 1.5) -- (1.25, 1.0);
    \node[right, purple] at (1.25, 1.25) {$t_j$};
    
    \filldraw[fill=white] (1.25, 1.5) circle (2pt); 
    \filldraw[fill=black] (1.25, 1.0) circle (2pt); 
\end{scope}

\begin{scope}[shift={(5,0)}]
    \foreach \y/\label in {0/1, 0.5/2, 2.0/n} {
        \draw [-{Stealth[scale=0.6]}] (0,\y) -- (1,\y);
        \draw (1,\y) -- (2.5,\y);
        \node[left] at (0,\y) {$\label$};
    }
    
    \node[left] at (0, 1.0) {$i_j$};
    \node[left] at (0, 1.5) {$i_{j+1}$};
    
    \draw [-{Stealth[scale=0.6]}] (0, 1.5) -- (1.0, 1.5);
    \draw [-{Stealth[scale=0.6]}] (1.5, 1.0) -- (2.5, 1.0);
    
    \draw [-{Stealth[scale=0.6]}] (0, 1.0) -- (1.0, 1.0);
    \draw [-{Stealth[scale=0.6]}] (1.5, 1.5) -- (2.5, 1.5);


\draw[thick] (1.0,1.5) to [out=0,in=180] (1.5,1.0);
\draw[color=white, line width=5] (1.0,1.0) to [out=0,in=180] (1.5,1.5);
\draw[thick] (1.0,1.0) to [out=0,in=180] (1.5,1.5);

    \node[below right, purple, scale=0.8] at (1.4, 1.8) {$-1$};

    \draw[thick, dotted] (1.25, 0.65) -- (1.25, 0.85);
    \draw[thick, dotted] (1.25, 1.65) -- (1.25, 1.85);
\end{scope}

\end{tikzpicture}
\end{center}

Unless otherwise indicated, the weight of an edge is 1. Concatenating the chips for $g_1, \dots, g_\ell$ from left to right, we obtain $G_{v, \bw}$. We label the wires on the far left and far right of $G_{v, \bw}$ with $1, \dots, n$ from bottom to top. See Figure~\ref{fig:3d plabic ex} for $G_{v, \bw}$ for Example~\ref{ex:running ex for 3d}.

\begin{figure}
    \centering
\begin{tikzpicture}[scale=1.5, line width=1pt]

    \foreach \y/\label in {0/1, 0.5/2, 1.0/3, 1.5/4} {
        \draw [-{Stealth[scale=0.6]}] (0,\y) -- (1,\y);
        \draw (1,\y) -- (1.5,\y);
        \node[left] at (0,\y) {$\label$};
        \node[right,purple] at (2.5,\y) {$\label$};
    }
    
     \draw (1.5,0) -- (2.5,0);
       \draw (2.0,0.5) -- (2.5,0.5);
      \draw [-{Stealth[scale=0.6]}] (2.0,1) -- (2.5,1);
       \draw (1.5,1.5) -- (2.5,1.5);

    \filldraw[fill=black] (0.25, 0.0) circle (2pt); 
       \draw [-{Stealth}](0.25, 0.5) -- (0.25, 0.0);
    \node[right, purple] at (0.25, 0.25) {$t_1$};
     \filldraw[fill=white] (0.25, 0.5) circle (2pt);
     
    \filldraw[fill=black] (0.75, 0.5) circle (2pt); 
       \draw[-{Stealth}] (0.75, 1.0) -- (0.75, 0.5);
    \node[right, purple] at (0.75, 0.75) {$t_2$};
    \filldraw[fill=white] (0.75, 1.0) circle (2pt); 

   
    \filldraw[fill=black] (1.25, 1.0) circle (2pt); 
       \draw[-{Stealth}] (1.25, 1.5) -- (1.25, 1.0);
    \node[right, purple] at (1.25, 1.25) {$t_3$};
     \filldraw[fill=white] (1.25, 1.5) circle (2pt); 

   
    \filldraw[fill=black] (2.25, 0.0) circle (2pt); 
       \draw[-{Stealth}] (2.25, 0.5) -- (2.25, 0.0);
    \node[right, purple] at (2.25, 0.25) {$t_5$};
     \filldraw[fill=white] (2.25, 0.5) circle (2pt); 


   \draw[thick] (1.5,1.0) to [out=0,in=180] (2.0,0.5);
   \draw[color=white, line width=5] (1.5,0.5) to [out=0,in=180] (2.0,1.0);
   \draw[thick] (1.5,0.5) to [out=0,in=180] (2.0,1.0);
   \node[right, purple] at (1.9,1.15) {$-1$};

\end{tikzpicture}

    \caption{The graph $G_{v, \bw}$ for $v= 1324$ and $\bw=s_1s_2s_3s_2s_1$} 
    \label{fig:3d plabic ex}
\end{figure}

The Lindstr\"om-Gessel-Viennot lemma \cite{lindstrom,GV} 
implies the MR matrix and its flag minors can be read off of paths in $G_{v, \bw}$. We state this more precisely below.

\begin{lemma}\label{lem:minors from paths signs}
   Let $g=(m_{ij})$ be the MR matrix for $(v, \bw)$. Then
   \[m_{ij} = \sum_{\substack{\text{paths}\\p:i \to j}} \prod_{e \in p} \wt(e),\]
   where the sum is over all directed paths in $G_{v,\bw}$ from sink $i$ to source $j$ and $\wt(e)$ is the weight of edge $e$.

For $I \in \binom{[n]}{k}$ the flag minor $\Delta_I(g)$ is equal to
\begin{equation}\label{eq:minors from paths signs} \Delta_I(g)= \sum_{P: I \to [k]} \sigma(P) \prod_{e \in P} \wt(e),\end{equation}
where the sum is over non-intersecting (NI) path collections $P$ from sources $I$ to sinks $[k]$ and $\sigma(P)$ is the sign of the permutation that $P$ defines.
\end{lemma}

In principle, \eqref{eq:minors from paths signs} could involve nontrivial cancellation, and two distinct path collections could have the same weight. As the next lemma shows, this is not the case.

\begin{lemma}\label{lem: minors from paths}
    Let $\bw$ be a reduced expression for $w$. For $I \in \binom{[n]}{k}$, 
    \begin{equation}\label{eq:minors from paths}\Delta_I(g) = \sum_{P: I \to [k]} \prod_{b \in P} \wt(b),\end{equation}
    where the sum is over NI path collections $P:I \to [k]$ in $G_{v, \bw}$ and the product is over the bridge edges $b$ in $P$.
    Moreover, on the right hand side, no two terms are the same.
\end{lemma}

\begin{proof} We use the notation $\wt(P):= \prod_{e \in P} \wt(e)$.
We will show that in $G_{v, \bw}$, 
\begin{enumerate}[(a)]
    \item if $P,P': I \to K$ are distinct NI path collections, $\wt(P) \neq \pm \wt(P') $;
    \item if $P: I \to [k]$ is an NI path collection, $\sigma(P)\wt(P)$ is equal to $\prod_{b\in P} \wt(b)$ where the product is over the bridges $b$ in $P$.
\end{enumerate}

We note that as soon as we establish (a), we show that $\Delta_I(g)$ is identically zero if and only if there are no NI path collections $P: I \to [k]$. Indeed, since distinct path collections have different weights, the formula for $\Delta_I(g)$ in \eqref{eq:minors from paths signs} is the zero polynomial if and only if there are no path collections from $I$ to $[k]$.

We proceed by induction on $\ell(w)$. The base case is $w=e$, in which case $v=e$ and $G_{v, \bw}$ consists of $n$ horizontal wires. Items (a) and (b) are straightforward in this case.

For the inductive step, we have $\bw = s_{i_1} \cdots s_{i_\ell}$. Say $s_{i_1}= s_j$. Let $\bw'= s_{i_2} \dots s_{i_\ell}$; this is a reduced expression for the permutation $w':=s_jw$. We have two cases.

\noindent\textbf{Case I: $1 \in J^+$}. Let $v'= s_{j} v$. We note that $v'<v, w'$. By induction, (a) and (b) hold for $G_{v', \bw'}$. The graph $G_{v, \bw}$ is obtained from $G_{v', \bw'}$ by adding a crossing on the left between wires $j$ and $j+1$. There is a bijection between NI path collections $P: I \to K$ in $G_{v, \bw}$ and NI path collections $P': s_j(I) \to K$ in $G_{v', \bw'}$ which preserves weights up to sign. So (a) holds for $G_{v, \bw}$ as well.

For (b): Consider an NI path collection $P:I \to [k]$ and the corresponding collection $P'$ in $G_{v', \bw'}$. Note
\[\prod_{b \in P} \wt(b) = \prod_{b \in P'} \wt(b) = \sigma(P')\wt(P')\]
where the last equality is by induction. We have that $\sigma(P) \wt(P) = \sigma(P') \wt(P')$ if $j, j+1 \notin I$, if $j+1 \in I, j \notin I$ and if $j, j+1 \in I$ (in this case, $\sigma(P)= -\sigma(P')$ and $\wt(P)= -\wt(P')$). We argue that the final situation, where $j \in I$ and $j+1 \notin I$ cannot occur. Indeed, we have $\Delta_I(g)$ is not identically zero, so $I=u[k]$ for some $u \in [v,w]$. Any reduced subexpression for $u$ in $\bw$ must contain a reduced expression for $v$, and thus must contain the first letter of $\bw$ (otherwise, there would be a reduced expression for $v$ in $\bw'$, contradicting the assumption that $1 \in J^+$). In other words, $s_ju<u$. This implies that $j+1$ appears to the left of $j$ in $u$, so $u[k]$ can't contain $j$ and not $j+1$.

\noindent\textbf{Case II: $1 \in J^\circ$.} In this case, $v \leq w'$. By induction, (a) and (b) hold for $G_{v, \bw'}$. The graph $G_{v, \bw}$ is obtained from $G_{v, \bw'}$ by adding a bridge $b$ of weight $t_1$ on the left between the wires at heights $j$ and $j+1$.

For (a): consider two NI path collections $P:I \to K$ and $Q:I \to K$ in $G_{v, \bw}$. If one uses the bridge $b$ and the other does not, then clearly $\wt(P) \neq \pm \wt(Q)$. If neither uses $b$, then $P, Q$ are also NI path collections in $G_{v, \bw'}$, so by induction $\wt(P) \neq \pm \wt(Q)$. If both use $b$, then ``chopping off" the leftmost chip of $G_{v, \bw}$ turns $P, Q$ into NI path collections $P',Q': s_j(I) \to K$ in $G_{v, \bw'}$. We have $\wt(P)= t_1\wt(P') \neq \pm t_1\wt(Q') = \pm \wt(Q)$ where the middle inequality is by induction. 

For (b): Consider an NI path collection $P: I \to [k]$ in $G_{v, \bw}$ and let $P':I' \to [k]$ be the NI path collection in $G_{v, \bw'}$ obtained by ``chopping off" the leftmost chip of $G_{v, \bw}$. By inspection, $\sigma(P) = \sigma(P')$ and $\wt(P)$ is either equal to $\wt(P')$ or is equal to $t_1 \wt(P')$. In either case, $\sigma(P') \wt(P')= \prod_{b \in P'} \wt(b)$ implies $\sigma(P) \wt(P)= \prod_{b \in P} \wt(b)$.

\end{proof}

\subsection{Combinatorics of Bruhat intervals}

We refer to \cite{BB05} for all details on Bruhat order on $S_n$.  Since the only intervals we will consider in this paper are intervals in the Bruhat order, we will henceforth just say ``intervals." Following \cite{BB05}, a $k$-crown is the face poset of a $k$-gon. In particular, the interval $[e,s_1s_2s_1]$ is a 2-crown.

We recall the definition of Bruhat interval polytopes following \cite[Definition 7.8]{TW15}. For $I \subset [n]$, we use the notation $e_I := \sum_{i \in I} e_i$ for the indicator vector of $I$ in $\R^n$.

\begin{definition}\label{defn: bruhat interval polytope}
The \emph{Bruhat interval polytope} $P(v,w)$ is
$$
P(v,w)=\Conv\left((n-1)e_{u(1)}+(n-2)e_{u(2)}+\ldots+e_{u(n-1)}\ :\ u\in [v,w]\right) \subset \R^n.
$$
\end{definition}

More conceptually, recall that $S_n$ acts on $\R^n$ by $u \cdot (\beta_1, \dots, \beta_n) = (\beta_{u^{-1}(1)}, \dots, \beta_{u^{-1}(n)})$. We have that 
\[P(v,w) = \Conv(u \cdot (n-1, n-2, \dots, 1, 0): u \in [v,w]).\]
Note that 
$$
(n-1)e_{u(1)}+(n-2)e_{u(2)}+\ldots+e_{u(n-1)}
=\sum_{k=1}^{n-1}e_{u[k]},
$$
and in fact we have a Minkowski sum decomposition \cite[Proposition 2.8]{TW15}
$$
P(v,w)=P_{1}(v,w)+\ldots+P_n(v,w)
$$
where 
$$
P_k(v,w)=\Conv\left(e_{u[k]}\ :\ u\in [v,w]\right) = \Conv\left(e_{I}\ :\ I \in \CM_k(v,w)\right)
$$
are the positroid polytopes for the constituents of $[v,w]$. 

\begin{remark}
The vertices of $P(v,w)$ are given by
$(n-1)e_{u(1)}+(n-2)e_{u(2)}+\ldots+e_{u(n-1)}$. Since $e_{u(1)}+\ldots+e_{u(n)}=e_1+\ldots+e_n$, we can instead consider the shifted Bruhat interval polytope
$P(v,w)+e_1+\ldots+e_n$ with vertices
$ne_{u(1)}+(n-1)e_{u(2)}+\ldots+2e_{u(n-1)}+e_n$. The latter convention was used in \cite{TW15}.
\end{remark}



\subsection{A lemma on posets}
Let $P$ be a poset and let $\Int(P)$ be the set of intervals of $P$, together with the empty set. We consider $\Int(P)$ to be a poset with respect to containment, so $\emptyset$ is the minimal element of $\Int(P)$. The following lemma will be useful to us.

\begin{lemma}\label{lem:poset-lattice-interval}
    A finite poset $P$ is a lattice if and only if $\Int(P)$ is a lattice.
\end{lemma}

\begin{proof}
    Suppose $P$ is a lattice with meet $\wedge$ and join $\vee$. Consider $[x,p], [y, q] \in \Int(P)$. By the definition of meet and join, $[a,b]$ contains $[x, p]$ and $[y,q]$ if and only if $[a,b]$ contains $[x \wedge y, p \vee q]$. So $[x \wedge y, p \vee q]$ is the least upper bound of $[x,p]$ and $[y,q]$. A similar argument shows that $[x \vee y, p \wedge q]$ is the greatest lower bound of $[x,y]$ and $[p,q]$. The least upper bound of $\emptyset$ and $[x,y]$ is $[x,y]$, and the greatest lower bound is $\emptyset$. Thus, $\Int(P)$ is a lattice.

    If $\Int(P)$ is a lattice, then it has a maximal element. This maximal element contains all intervals $[x,x]$ so must be the poset $P$. That is, $P$ is itself an interval and so must have a minimal element $\hat{0}$ and a maximal element $\hat{1}$. The poset $P$ is isomorphic to the interval $[[\hat{0}, \hat{0}], [\hat{0}, \hat{1}]]$ in $\Int(P)$. Since subintervals of lattices are lattices, $P$ is a lattice.
\end{proof}

\section{Toric Richardson varieties}\label{sec:characterization} In this section, we characterize the toric Richardsons $\overline{R}_{v,w}$ in various ways.

\subsection{On open Richardsons which are tori}
We begin with some helpful results on open Richardsons $R_{v,w}$ which are tori. 

\begin{lemma}
\label{lem: 2 crown}
The open Richardson variety $R_{v,w}$ is a torus if and only if $[v,w]$ contains no 2-crowns.
\end{lemma}

\begin{proof}

It follows from the Deodhar decomposition that the $\sR$-polynomial can be written as
$$
\sR_{v,w}(q)=\sum_{\sigma}(q-1)^{a(\sigma)}q^{b(\sigma)}
$$
where the sum is over distinguished subexpressions of a particular reduced word $\bw$ for $w$. Recall that $a(\sigma) + 2b(\sigma) = \ell(w)-\ell(v)= d$. Let $\sm$ be the number of codimension 1 Deodhar components (with $a(\sigma)=d-2,b(\sigma)=1$). Then 
$$
\sR_{v,w}(q)=(q-1)^d+\sm(q-1)^{d-2}q+\ldots=q^d-(d-\sm)q^{d-1}+\ldots
$$

The number $\sm$ is also the number of mutable variables in a seed for the cluster structure on $R_{v,w}$. By Corollary~\ref{cor: all frozen}, $R_{v,w}$ is a torus if and only if there are no mutable variables, that is $\sm=0$. So $R_{v,w}$ is a torus if and only if the coefficient of $q^{d-1}$ in $\sR_{v,w}(q)$ is $(-d)$. By \cite[Theorem 6.3]{Brenti} (see also \cite[Section 5, Exercise 36]{BB05}) this is equivalent to the fact that $[v,w]$ has no 2-crowns.
\end{proof}

\begin{corollary}
\label{cor: subintervals tori}
Assume that $R_{v,w}$ is a torus. Then for all $v\le v'\le w'\le w$ the open Richardson  $R_{v',w'}$ is also a torus.
\end{corollary}

\begin{proof}
Since $R_{v,w}$ is a torus, by Lemma \ref{lem: 2 crown} the interval $[v,w]$ does not contain 2-crowns. Therefore $[v',w']$ does not contain 2-crowns, and by applying Lemma \ref{lem: 2 crown} we conclude that $R_{v',w'}$ is a torus. 
\end{proof}

We next turn to flag minors when $R_{v,w}$ is a torus. Recall from \cite[Proposition 5.12]{GKSS2} that for all $v\le u\le w$ we have a splicing map
$$
\Psi_{v,u,w}: R_{v,u}\times R_{u,w}\to R_{v,w}
$$
which defines an isomorphism between $R_{v,u}\times R_{u,w}$ and a principal open subset $U(v,u,w)\subset R_{v,w}$ defined by the inequalities 
$$
U(v,u,w)=\{\Delta_{u[1]} \Delta_{u[2]} \cdots \Delta_{u[n-1]}\neq 0\}.
$$

\begin{lemma}
\label{lem: splitting}
Assume $R_{v,w}$ is a torus $\TT$ and $v\le u\le w$. Then:

a) The map $\Psi_{v,u,w}$ is an isomorphism and $U(v,u,w)=R_{v,w}$.

b) For any choice of multiplicative structure on $R_{v, w}$, the minors $\Delta_{u[k]}$ are nonzero Laurent monomials.
\end{lemma}

\begin{proof}
a) By Corollary \ref{cor: subintervals tori} we get that $R_{v,u}$ and $R_{u,w}$ are tori of dimensions $\ell(u)-\ell(v)$ and $\ell(w)-\ell(u)$ respectively. Therefore $R_{v,u}\times R_{u,w}\simeq U(v,u,w)$ is also a torus of dimension $\ell(w)-\ell(v)$.

On the other hand, $U(v,u,w)$ is an open subset in the torus $R_{v,w}$ of the same dimension, so by Lemma \ref{lem: open torus} we get $U(v,u,w)=R_{v,w}$.

b) All the minors $\Delta_{u[k]}$ are nonzero on $U(v,u,w)$. Since $U(v,u,w)=R_{v,w}$, we conclude
that $\Delta_{u[k]}$ are invertible functions on the torus, that is, nonzero monomials.
\end{proof}

The above lemma is true regardless of the multiplicative structure we fix on $R_{v,w}$. If we fix a particular multiplicative structure, we obtain additional information. The next lemma discusses flag minors for the multiplicative structure given by Proposition~\ref{prop:MR-param}. That is, the coordinates of $R_{v,w}$ are exactly the MR parameters $\{t_j\}_{j \in J^\circ}$. 

\begin{lemma}\label{lem:minors in MR-param}
    Suppose $R_{v,w}$ is a torus, with multiplicative structure given by Proposition~\ref{prop:MR-param}. Fix $\bw$ a reduced expression for $w$. 
    
    For any $I \in \binom{[n]}{k}$, either there are no NI path collections $P:I \to [k]$ in $G_{v, \bw}$ and $\Delta_I(g)=0$, or there is a unique NI path collection $P :I \to [k]$ and
    \[\Delta_I(g) = \prod_{b \in P} \wt(b)\]
    where the product is over all bridge edges in $P$.

    In particular, the nonzero flag minors of the MR matrix $g$ are square-free monic monomials in the MR parameters $\{t_j\}_{j \in J^\circ}$.
\end{lemma}

\begin{proof}
    By \eqref{eq:minors from paths}, $\Delta_I(g)$ is the sum of distinct monic monomials in the MR parameters, one for each NI path collection $P:I \to [k]$. So $\Delta_I(g)$ is identically zero if and only if there are no such path collections. If $\Delta_I(g)$ is not identically zero, then by Theorem~\ref{thm: flag minor richardson}, $I=u[k]$ for some $u \in [v,w]$ and by Lemma~\ref{lem: splitting}, $\Delta_I(g)$ is a nonzero Laurent monomial in the MR parameters. This is possible only if there is a unique NI path collection $P:I \to [k]$.
\end{proof}

\subsection{On closed Richardsons which are toric}

In this section, we characterize toric closed Richardson varieties. We start by ruling out some Richardsons.

\begin{lemma}[{\cite[Corollary 3.4]{Can}}]
\label{lem: converse}
Assume that $\overline{R}_{v,w}$ is a toric variety with respect to some torus action. Then $[v,w]$ does not contain 2-crowns.
\end{lemma}

\begin{proof}
We describe the argument from \cite{Can} for the reader's convenience. If $\overline{R}_{v,w}$ is a toric variety with respect to the action of some torus $\TT'$ then by \cite[Theorem 3.1]{Can} for all $[v',w']$ the subvariety $\overline{R}_{v',w'}$ is a closed irreducible $\TT'$-invariant subvariety of $\overline{R}_{v,w}$. Therefore it corresponds to the closure of some face of the corresponding moment polytope, and is a toric variety itself. 

On the other hand, if $[v',w']$ is a 2-crown then one can check as in \cite[Corollary 3.3]{Can} that $\overline{R}_{v',w'}$ cannot be a toric variety. In particular, there is no polytope whose faces are in inclusion-preserving bijection with the subintervals of a 2-crown. So if $\overline{R}_{v,w}$ is toric, $[v,w]$ does not contain 2-crowns.
\end{proof}

Now we show that all $\overline{R}_{v,w}$ that are not ruled out by Lemma~\ref{lem: converse} are in fact toric.

\begin{theorem}
\label{thm: toric}
We have that $\overline{R}_{v,w}$ is a toric variety if and only if $R_{v,w}$ is a torus.
\end{theorem}

\begin{proof}
By Lemma~\ref{lem: 2 crown}, $R_{v,w}$ is a torus if and only if $[v,w]$ does not contain 2-crowns. If $R_{v,w}$ is not a torus, then $[v,w]$ contains a 2-crown and Lemma~\ref{lem: converse} shows that $\overline{R}_{v,w}$ is not toric. We now show the other direction, so assume $R_{v,w}$ is a torus.

Let us choose some coordinates $x_1,\ldots,x_d$ on $R_{v,w}\simeq (\C^\times)^d$. We claim that we may choose these coordinates so that each flag minor $\Delta_I$ is either identically zero or is a nonzero monic Laurent monomial in $x_i$. Indeed, if $I\neq u[k]$ for some $u\in [v,w]$ then $\Delta_I=0$ by Theorem \ref{thm: flag minor richardson}. Otherwise $\Delta_I$ is a nonzero Laurent monomial by Lemma \ref{lem: splitting}, and Lemma~\ref{lem:minors in MR-param} shows that we may assume the Laurent monomials are monic. When $\Delta_I \neq 0$, we can write
\begin{equation}
\label{eq: Delta as monomial}
\Delta_I= x_1^{m_{I,1}}\cdots x_d^{m_{I,d}},\ m_{I,j}\in \Z.
\end{equation}
Now consider the composition of maps:
$$
\begin{tikzcd}
\Fl(n) \arrow{r} \arrow[swap]{dr}{\jmath}& \prod_{k=1}^{n-1}\Gr(k,n) 
\arrow{d}{\alpha}& \\
& \prod_{k=1}^{n-1}\PP^{\binom{n}{k}-1}
\end{tikzcd}
$$
where the horizontal map sends a flag to a collection of its subspaces, and $\alpha$ is the Pl\"ucker embedding. 
The image of $\jmath$ is closed and is defined by the incidence Pl\"ucker relations \eqref{eq: incidence Plucker}. Now consider the composition
$$
\begin{tikzcd}
(\C^{\times})^d\simeq R_{v,w} \arrow{dr}{\iota} \arrow[hookrightarrow]{r}& \Fl(n)\arrow{d}{\jmath} \\
 & \prod_{k=1}^{n-1}\PP^{\binom{n}{k}-1}.
\end{tikzcd}
$$
We can think of the closed Richardson variety $\overline{R}_{v,w}$ as the closure of the image of $\iota$, as the image of $\iota$ is automatically contained in the closed subset $\jmath(\Fl(n))$. By the above, $\iota$ is given in coordinates by monic Laurent monomials in $x_i$, hence $$\overline{R}_{v,w}\simeq \overline{\iota(R_{v,w})}$$ is a toric variety as in \cite{Cox}.
\end{proof}

\begin{definition}\label{def: richardson w mult struc}
We define the algebraic torus $\TT_{v,w}\simeq (\C^{\times})^d$ as the variety $R_{v,w}$ with some choice of the multiplicative structure which we fix for the remainder of the section.  
\end{definition}

We can rephrase Theorem \ref{thm: toric} as follows.

\begin{corollary}
When $R_{v,w}$ is a torus, the action of $\TT_{v,w}$ on $R_{v,w}$ extends to an action of $\TT_{v,w}$ on $\overline{R}_{v,w}$.    
\end{corollary}

Recall the decomposition
$$\overline{R}_{v,w}=\bigsqcup_{v\le v'\le w'\le w}R_{v',w'}$$
of closed Richardson varieties into open Richardson varieties. We show that when $\overline{R}_{v,w}$ is toric, this coincides with the decomposition of $\overline{R}_{v,w}$ into $\TT_{v,w}$-orbits.

\begin{lemma}
\label{lem: orbits}
Assume $R_{v,w}$ is a torus and define the action of $\TT_{v,w}$ on $\overline{R}_{v,w}$ as in Theorem \ref{thm: toric}. Then for all $v\le v'\le w'\le w$ the subvariety $R_{v',w'}\subset \overline{R}_{v,w}$ is a $\TT_{v,w}$-orbit.
\end{lemma}

\begin{proof}
By \cite[Theorem 3.1]{Can} if the closed Richardson variety $\overline{R}_{v,w}$ is toric for some torus $\TT$ then all closed Richardson varieties  $\overline{R}_{v',w'}$ for subintervals are $\TT$-invariant subsets. Since
$$
R_{v',w'}=\overline{R}_{v',w'}\setminus \bigcup_{[v'',w'']\subset [v',w']}\overline{R}_{v'',w''}
$$
we conclude that the open Richardson variety $R_{v',w'}$ is also $\TT$-invariant.

On the other hand, by Corollary \ref{cor: subintervals tori} the variety $R_{v',w'}$ is a torus itself, and action of the torus $\TT$ on the torus $R_{v',w'}$ can have either one or infinitely many orbits.

Since $\TT_{v,w}$ has finitely many orbits in $\overline{R}_{v,w}$, we conclude that $R_{v',w'}$ 
is a $\TT_{v,w}$-orbit.
\end{proof}

Now that we know the torus orbits, we can describe the moment polytope of toric Richardsons. We will discuss the moment polytope further in Section~\ref{subsec: polytopes}.

\begin{corollary}
\label{cor: face lattice} Suppose $\overline{R}_{v,w}$ is toric and let $\widehat{P}(v,w)$ denote its moment polytope. Then the $s$-dimensional faces of $\widehat{P}(v,w)$ are in bijection with the subintervals $[v',w']\subset [v,w]$ with $\ell(w')-\ell(v')=s$. The face poset of $\widehat{P}(v,w)$ is isomorphic to the poset of subintervals of $[v,w]$ ordered by inclusion.
\end{corollary}

\begin{corollary}\label{cor: combinatorial type}
    The combinatorial type of $\widehat{P}(v,w)$ is determined by the poset structure of $[v,w]$.
\end{corollary}

\subsection{Positroid projections}

In this subsection, we collect a few results regarding the projections of open Richardsons which are tori, and closed Richardsons which are toric, to the Grassmannian. By \cite[Theorem 5.1]{KLS09}, $\pi_k(\overline{R}_{v,w})$ is always a closed positroid variety, but $\pi_k(R_{v,w})$ may not be an open positroid variety. Among other things, we show that when $R_{v,w}$ is a torus, $\pi_k(R_{v,w})$ is indeed an open positroid variety, and is also a torus. When $\overline{R}_{v,w}$ is toric, $\pi_k(\overline{R}_{v,w})$ is toric with respect to the standard torus.

Consider the projections $\pi_k:\Fl(n)\to \Gr(k,n)$ from the complete flag variety to the Grassmannian of $k$-planes. Recall the positroid 
$\CM_k(v,w):=\{u[k]: v\le u\le w\}$ . 
We denote by $\Pi^{\circ}_k(v,w):=\Pi^{\circ}_{\CM_k(v,w)}$ (resp. $\Pi_k(v,w):=\Pi_{\CM_k(v,w)}$)  the corresponding open (resp. closed) positroid variety and by $G_k(v,w)$ the corresponding plabic graph.

\begin{lemma}
\label{lem: contained in positroid}
Assume that $R_{v,w}$ is a torus. Then
$\pi_k(R_{v,w})$ is a dense subset of $\Pi^{\circ}_k(v,w)$.
\end{lemma}

\begin{proof}
The containment $\pi_k(R_{v,w})\subset \Pi^{\circ}_k(v,w)$ is because all points in $\pi_k(R_{v,w})$ have $\Delta_I \neq 0$ if and only if $I= u[k]$ for some $u \in [v,w]$, by Theorem~\ref{thm: flag minor richardson} and Lemma~\ref{lem: splitting}. This implies that they are contained in $\Pi^{\circ}_k(v,w)$. On the other hand, by \cite[Theorem 5.1]{KLS13} we have $\Pi_k(v,w)=\pi_k(\overline{R}_{v,w})$, so  $\pi_k(R_{v,w})$ is dense in $\Pi_k(v,w)$. By the above, it is also dense in $\Pi^{\circ}_k(v,w)$.
\end{proof}

\begin{lemma}
\label{lem: projection toric}
Assume $R_{v,w}$ is a torus. Then $\Pi_k(v,w)=\pi_k(\overline{R}_{v,w})$ is a toric variety and the projection $\pi_k$ is $\TT_{v,w}$-equivariant. 
\end{lemma}

\begin{proof}
This is immediate from the proof of Theorem \ref{thm: toric}: all flag minors for $I \in \binom{[n]}{k}$ are either zero or Laurent monomials in the coordinates of $\TT_{v,w}$, and these flag minors are exactly the Pl\"ucker coordinates in the projection.
\end{proof}

\begin{lemma}
\label{lem: disjoint projections}
Assume that $R_{v,w}$ is a torus and $[v',w']\subsetneq [v,w]$. Then either $\pi_k(R_{v',w'})\cap \pi_k(R_{v,w})=\emptyset$ or $\pi_k(R_{v',w'})=\pi_k(R_{v,w})$.
\end{lemma}

\begin{proof}
By Lemma \ref{lem: projection toric} the projection $\pi_k$ is $\TT_{v,w}$-equivariant. Assume that a point $p$ is contained in the intersection $\pi_k(R_{v',w'})\cap \pi_k(R_{v,w})$. 

Since $R_{v', w'}$ is also a torus by Corollary~\ref{cor: subintervals tori}, Lemma \ref{lem: contained in positroid} implies that the image $\pi_k(R_{v',w'})$ is contained in $\Pi^{\circ}_{k}(v',w')$. Open positroid varieties labeled by different positroids are disjoint. Therefore $\CM_k(v',w')=\CM_k(v,w)$ and $\pi_k(R_{v',w'})$ is dense in $\Pi^{\circ}_k(v,w)$. In particular, we have
$$
\dim \Pi_k(v,w)=\dim \pi_k(R_{v',w'})=\dim \pi_k(R_{v,w}).
$$

The $\TT_{v,w}$-orbit $\mathbf{O}$ of $p$ is a torus which is contained in $\Pi_k(v,w)$ by Lemma \ref{lem: contained in positroid} and contains both $\pi_k(R_{v',w'})$ and $\pi_k(R_{v,w})$ by Lemma \ref{lem: orbits}. The above equality of dimensions implies that $\mathbf{O}$ is also the same dimension as $\pi_k(R_{v',w'})$, and $\pi_k(R_{v,w})$. By Lemma \ref{lem: open torus} we conclude that $\mathbf{O}=\pi_k(R_{v',w'})=\pi_k(R_{v,w})$.
\end{proof}

\begin{theorem}
\label{thm: positroid torus}
Assume that $R_{v,w}$ is a torus. Then:
\begin{itemize}
\item[(a)] We have $\pi_k(R_{v,w})=\Pi^{\circ}_k(v,w)$.
\item[(b)] The open positroid $\Pi^{\circ}_k(v,w)$ is a torus and the closed positroid variety is toric with respect to $\TT_{v,w}$.
\item[(c)] The plabic graph $G_{k}(v,w)$  is a forest. 
\item[(d)] The closed positroid variety $\Pi_k(v,w)=\pi_k(\overline{R}_{v,w})$ is toric with respect to the standard torus $T$. 
\end{itemize}
\end{theorem}

\begin{proof}
(a) We have $\overline{R}_{v,w}=\bigsqcup_{v\le v'\le w'\le w}R_{x,y}$, so 
$$
\pi_k(\overline{R}_{v,w})=\bigcup_{v\le v'\le w'\le w}\pi_k(R_{v',w'}).
$$
By Lemma \ref{lem: contained in positroid}, $\pi_k(R_{v',w'})$ is contained in the positroid $\Pi^{\circ}_{k}(v',w')$ which either coincides with $\Pi^{\circ}_k(v,w)$ or is disjoint from it. Since $\pi_k(\overline{R}_{v,w})$ contains $\Pi^{\circ}_k(v,w)$, we can write
$$
\Pi^{\circ}_k(v,w)=\bigcup_{\stackrel{v\le v'\le w'\le w}{\CM_k(v',w')=\CM_k(v,w)}}\pi_k(R_{v',w'})
$$
By Lemma \ref{lem: disjoint projections} the projections $\pi_k(R_{v',w'})$ either coincide or do not intersect at all. Since  $\Pi^{\circ}_k(v,w)$ is irreducible, they cannot be disjoint and therefore $\pi_k(R_{v,w})=\Pi^{\circ}_k(v,w)$, as desired. 

(b) Part (a) implies that $\Pi^{\circ}_k(v,w)$ is a torus, as all nonzero Pl\"ucker coordinates are monomials in the coordinates of the torus $R_{v,w} = \TT_{v,w}$. The closed positroid is the closure of the image of a monomial map $\TT_{v,w} \to \PP^{\binom{n}{k}-1}$ so is toric with respect to $\TT_{v,w}$.

(c) The  plabic graph $G_k(v,w)$ is a forest by Lemma \ref{lem: forest}. This implies (d) because in this case, the dimension of $\Pi_k(v,w)$ is the same as the dimension of the $T$-moment polytope $P_{k}(v,w)$.
\end{proof}



\section{Moment polytopes of toric Richardsons}\label{subsec: polytopes}


In this section, we focus on the moment polytope $\widehat{P}(v,w)$ of the toric Richardson variety $\overline{R}_{v,w}$. Corollary~\ref{cor: face lattice} described the face lattice of this polytope, which is the poset of subintervals of $[v,w]$. Here, we give more detail on its geometry.

\subsection{The moment polytopes as Minkowski sums}

We fix $R_{v,w}$ which is a torus. As in Definition~\ref{def: richardson w mult struc}, we use $\TT_{v,w}$ to denote $R_{v,w}$ with a fixed multiplicative structure. We use $x_i$ to denote the coordinates of $\TT_{v,w}$.
Recall that on $R_{v,w}$, if $I \in \CM_k(v,w)$, then by \eqref{eq: Delta as monomial}
\[\Delta_I= x_1^{m_{I,1}} \dots x_d^{m_{I,d}}\]
for some $m_I=(m_{I,1},\ldots,m_{I,d}) \in \Z^d$.

\begin{definition}
We define the polytope $\widehat{P}_{k}(v,w)$ as
\[\widehat{P}_{k}(v,w)= \Conv(m_I: I \in \CM_k(v,w)).\]
We call this the \emph{$k$th summand} for $[v,w]$, for reasons that will be clear momentarily.
\end{definition}

Recall that the vectors $m_I$ are defined up to a shift which might depend on $|I|=k$ but not on $I$. This means that the polytopes  $\widehat{P}_{k}(v,w)$ are also defined up to a shift.

\begin{theorem}
\label{thm: polytope} Suppose $R_{v,w}$ is a torus, so by Theorem~\ref{thm: positroid torus} $\pi_k(\overline{R}_{v,w})=\Pi_k(v,w)$ is toric.
\begin{enumerate}[a)]
    \item The moment polytope of $\Pi_k(v,w)$ with respect to the action of the torus $\TT_{v,w}$ is $\widehat{P}_{k}(v,w)$. 
    \item The moment polytope $\widehat{P}(v,w)$ of $\overline{R}_{v,w}$ is equal to the Minkowski sum $\sum_{k=1}^{n-1}\widehat{P}_{k}(v,w)$.
\end{enumerate}
\end{theorem}


\begin{proof}
Both parts follow from the proof of Theorem \ref{thm: toric}, so we only need to explain the appearance of the Minkowski sum. We consider the Segre embedding
$$
R_{v,w}\xrightarrow{\iota}\prod_{k=1}^{n-1}\PP^{\binom{n}{k}-1}\hookrightarrow \PP^{N},\ N=\prod_{k=1}^{n-1}\binom{n}{k}-1.
$$
The coordinates in the image of $R_{v,w}$ are either zero or monomials with exponent vectors
$$
m_{I_1}+ m_{I_2} + \dots + m_{I_{n-1}}
$$
where $I_k\in \CM_{k}(v,w)$. Therefore the moment polytope for $\overline{R}_{v,w}$ is given by
$$
\Conv\left(m_{I_1}+ m_{I_2} + \dots + m_{I_{n-1}}\ :\ I_k\in \CM_{k}(v,w)\right)
$$
$$
=\Conv\left(m_{I_1}\ :\ I_1\in \CM_{1}(v,w)\right)+\dots+\Conv\left(m_{I_{n-1}}\ :\ I_{n-1}\in \CM_{n-1}(v,w)\right)
$$
$$
=\widehat{P}_1(v,w)+\dots+\widehat{P}_{n-1}(v,w).
$$
Here we used the identity $\Conv(A+B)=\Conv(A)+\Conv(B)$.

\end{proof}


\begin{corollary}
The vertices of $\widehat{P}_k(v,w)$ are exactly the $m_I$ for $I\in \CM_{k}(v,w)$.
\end{corollary}

\begin{proof}
The vertices of  $\widehat{P}_k(v,w)$ correspond to the $\TT_{v,w}$-fixed points of $\Pi_k(v,w)=\pi_k(\overline{R}_{v,w})$. These are the same as $T$-fixed points, and given by coordinate flags for subsets $u[k],\ u\in [v,w]$, equivalently, by $I\in \CM_{k}(v,w)$.
\end{proof}




Recall from Definition~\ref{def: positroid polytope} the definition of the positroid polytope $P_k(v,w)$. This polytope lies in the affine hyperplane $\mathcal{H}_k$ in $\R^n$ consisting of vectors with coordinate sum $k$. Alternatively, we can think of $P_k(v,w)$
as the moment polytope of $\Pi_k(v,w)$ under the action of the standard torus $T$ where a nonzero minor $\Delta_I$ has weight $e_I$.

\begin{lemma}
\label{lem: projection}
There exist an integer matrix $C\in \Mat(d,n)$ and integer vectors $d^{(0)}, d^{(1)},\ldots,d^{(k)}\in \Z^n$ such that
$$
C\widehat{P}(v,w)+d^{(0)}=P(v,w),\ C\widehat{P}_k(v,w)+d^{(k)}=P_k(v,w)
$$
and
$$
Cm_{I}+d^{(k)}=e_{I}\quad \mathrm{for}\ I\in \CM_k(v,w).
$$
\end{lemma}

\begin{proof}
The action of the standard torus $T$ on $R_{v,w}$ defines a homomorphism $T\to \TT_{v,w}$ which corresponds to the map of lattices $C:\Z^{d}\to \Z^{n}$. The action of $T$ on $\overline{R}_{v,w}$  factors through the action of $\TT_{v,w}$, and hence we get a commutative diagram of moment maps
$$
\begin{tikzcd}
\overline{R}_{v,w} \arrow{d}{\mu_{\TT}} \arrow{dr}{\mu_T} & \\
\R^{d}\arrow{r}{C} & \R^{n}
\end{tikzcd}
$$
and $C\widehat{P}(v,w)$ agrees with $P(v,w)$ up to translation. The proof of the other statements is similar. 
\end{proof}
 

The next lemma shows that the $k$th summand $\widehat{P}_k(v,w)$ is affinely equivalent to $P_k(v,w)$, the positroid polytope for the $k$th constituent of $[v,w]$.

\begin{lemma}
\label{lem: transformation}
For each $k$ there exists an integer matrix $A^{(k)}\in \Mat(n,d)$ and an integer vector $b^{(k)}\in \Z^d$ such that $\widehat{P}_k(v,w)=A^{(k)}P_k(v,w)+b^{(k)}$. Equivalently, there exist vectors $a^{(k)}_1,\ldots,a^{(k)}_n$ and $b^{(k)}$ in $\Z^d$ such that for all $I \in \CM_k(v,w)$,
\begin{equation}
\label{eq: mI split}
m_I=b^{(k)}+\sum_{i\in I} a^{(k)}_i.
\end{equation}
This defines an affine equivalence between $P_{k}(v,w)$ and $\widehat{P}_k(v,w)$.
\end{lemma}

\begin{proof}
By Theorem \ref{thm: positroid torus} we get that $\Pi^{\circ}_k(v,w)$ is a torus and the corresponding plabic graph $G_k(v,w)$ is a forest. All cluster variables in the corresponding cluster algebra are frozen. These cluster variables are Pl\"ucker coordinates $\Delta_{I_1},\ldots,\Delta_{I_s}$ corresponding to the $s$ faces of $G_k(v,w)$ (the subsets $I_1,\dots, I_s$ are the elements of the \emph{Grassmann necklace}). These Pl\"ucker coordinates are algebraically independent. Note also that $\Pi^{\circ}_k(v,w)$ has dimension $s-1$, and this is the same as the dimension of the 
polytopes $P_k(v,w)$ and $\widehat{P}_k(v,w)$, as both polytopes are moment map images of $\Pi_k(v,w)$ with respect to the tori $T$ and $\TT_{v,w}$ respectively.

The nonzero Pl\"ucker coordinates $\{\Delta_I : I \in \CM_k(v,w)\}$ of $\Pi^{\circ}_k(v,w)$ are units; this follows for example from Theorem \ref{thm: positroid torus} (a). By \cite[Theorem 1.3]{GLS}, this implies that the nonzero Pl\"ucker coordinates are Laurent monomials in the frozen variables $\Delta_{I_1},\ldots,\Delta_{I_s}$:
$$
\Delta_I=\Delta_{I_1}^{\ell_{I,1}}\cdots\Delta_{I_s}^{\ell_{I,s}} \quad \text{for } \ell_{I,1},\ldots,\ell_{I,s}\in \Z.
$$
By comparing the weights of $T$ and $\TT_{v,w}$ we get
\begin{equation}
\label{eq: Grassmann necklace generates}
e_I=\ell_{I,1}e_{I_1}+\ldots+\ell_{I,s}e_{I_s}\ \quad \text{and} \quad m_I=\ell_{I,1}m_{I_1}+\ldots+\ell_{I,s}m_{I_s} \quad \text{for } I\in \CM_{k}(v,w).
\end{equation}

The points $e_{I_1},\ldots,e_{I_s}$ span a simplex in $\R^{n-1}$ (otherwise $P_k(v, w)$ would not be dimension $s-1$). So there exists an affine transformation $\beta\mapsto A^{(k)}\beta+b^{(k)}$ from $\Z^{n}$ to $\Z^d$ such that $A^{(k)}e_{I_j}+b^{(k)}=m_{I_j}$ for $j=1,\ldots,s$.
By \eqref{eq: Grassmann necklace generates} we get $A^{(k)}e_I+b^{(k)}=m_I$ for all $I\in \CM_{k}(v,w)$ and 
$$
A^{(k)}P_{k}(v,w)+b^{(k)}=\widehat{P}_k(v,w).
$$

Let $a^{(k)}_i$ be the columns of the matrix $A_k$, then
$$
A^{(k)}\left(\sum_{i\in I}e_i\right)+b^{(k)}=\sum_{i\in I}\left(A^{(k)}e_i\right)+b^{(k)}=\sum_{i\in I}a^{(k)}_i+b^{(k)},
$$
and the equation \eqref{eq: mI split} follows. More abstractly, we can say that the matrix $A^{(k)}$ defines a homomorphism $\TT_{v,w}\to T$ and the action of $\TT_{v,w}$ on $\Pi_k(v,w)$ factors  through the action of $T$. 

Finally, by Lemma \ref{lem: projection} for all $I\in \CM_{k}(v,w)$ we get 
$$
C(A^{(k)}e_I+b^{(k)})+d^{(k)}=Cm_I+d^{(k)}=e_I
$$ 
This implies that $C$ and $A^{(k)}$ define inverse affine transformations between $P_k(v,w)$ and $\widehat{P}_k(v,w)$, and hence these polytopes are affine equivalent.
\end{proof}

\begin{remark}
\label{rmk: shift}
Adding the same vector $a$ to all $a_i^{(k)}$ is equivalent to changing $b^{(k)}$ to $b^{(k)}+ka$. This will shift the polytope $\widehat{P}_k(v,w)$ by a constant vector. 
\end{remark}

\subsection{Vertices and edges}

We already know from Corollary~\ref{cor: face lattice} that the vertices of $\widehat{P}_{v,w}$ are in bijection with the $\TT_{v,w}$-fixed points of $\overline{R}_{v,w}$, which are exactly the 0-dimensional open Richardsons $R_{u,u}$ with $u \in [v,w]$. The next corollary expresses these vertices in terms of the exponent vectors $m_I$ and the vectors from Lemma~\ref{lem: transformation}.

\begin{corollary}
The vertices of the polytope $\widehat{P}(v,w)$ are given by
\begin{equation}
\label{eq: vertices}
X_{u}:=\sum_{k=1}^{n-1}m_{u[k]}=\sum_{k=1}^{n-1}\sum_{i=1}^{k}a^{(k)}_{u(i)}+\sum_{k=1}^{n-1}b^{(k)}.
\end{equation}
for $u \in [v,w]$.
\end{corollary}
\begin{proof}
By Corollary \ref{cor: face lattice} the vertices $X_u$ of $\widehat{P}(v,w)$ correspond to the coordinate flags 
$\mathcal{F}(u)=uB/B$
for permutations $u\in [v,w]$. The image of $\mathcal{F}(u)$ under the projection $\pi_k$ is the coordinate subspace   which corresponds to the subset $u[k]$ and to the vertex $m_{u[k]}$ in the polytope $\widehat{P}_{k}(v,w)$. Therefore by the proof of Theorem~\ref{thm: polytope} we get
$$
X_u=\mu(\mathcal{F}(u))=\sum_{k=1}^{n-1}\mu_k(\pi_k(\mathcal{F}(u)))=\sum_{k=1}^{n-1}m_{u[k]}.
$$
\end{proof}



Next, we study the compatibility between the edge vectors in different constituent polytopes $\widehat{P}_k(v,w)$, and the respective matrices $A^{(k)}$. Recall that any edge in the Minkowski sum of polytopes is parallel to an edge in one of the summands, or is the Minkowski sum of parallel edges in different summands.  The following result clarifies and strengthens this observation. For two vectors $m, m'$, we use the notation $\overline{m m'}$ for the difference $m'-m.$


\begin{lemma}
\label{lem: edge vectors 1}
Assume that $i<j$, $u'=u\cdot (i\ j)$ covers $u$ in Bruhat order, and $u,u' \in [v,w]$. Then for all $i\le k,k'<j$ the edge vectors   $\overrightarrow{m_{u[k]}m_{u'[k]}}$ in $\widehat{P}_{k}(v,w)$ and  $\overrightarrow{m_{u[k']}m_{u'[k']}}$ in $\widehat{P}_{k'}(v,w)$ are the same. 
\end{lemma}

\begin{proof}
It is sufficient to prove the statement for $k'=k+1<j$. Consider two subsets $$I=\{u(1),\ldots,u(i-1),u(i+1),\ldots,u(k)\}$$ and 
$$J=\{u(1),\ldots,u(i-1),u(i+1),\ldots,u(k),u(i),u(j),u(k+1)\}.$$ Note that $J\setminus I=\{u(i),u(j),u(k+1)\}$. The incidence Plucker relation \eqref{eq: incidence Plucker} for $I,J$ reads as
$$
\pm\Delta_{I\cup u(i)}\Delta_{J\setminus u(i)}\pm \Delta_{I\cup u(j)}\Delta_{J\setminus u(j)}\pm \Delta_{I\cup u(k+1)}\Delta_{J\setminus u(k+1)}=0.
$$
Note that $I\cup u(i)=u[k]$ and $J\setminus u(i)=u'[k+1]$ while
$I\cup u(j)=u'[k]$ and $J\setminus u(j)=u[k+1]$, so in particular both products $\Delta_{I\cup u(i)}\Delta_{J\setminus u(i)}$ and $\Delta_{I\cup u(j)}\Delta_{J\setminus u(j)}$ are nonzero monomials. Assume that the third product $\Delta_{I\cup u(k+1)}\Delta_{J\setminus u(k+1)}$ is also nonzero. By Lemma \ref{lem:minors in MR-param} we can assume that all three products are monic monomials, but then the relation is impossible. Therefore
the third term vanishes and we get
\begin{equation}
\label{eq: incidence plucker edge 2}
\Delta_{u[k]}\Delta_{u'[k+1]}=\Delta_{u'[k]}\Delta_{u[k+1]}.
\end{equation}

By comparing the weights of both sides of \eqref{eq: incidence plucker edge 2} we get the identity:
$$
\sum_{s\in I}a^{(k)}_s+a^{(k)}_{u(i)}+\sum_{s\in I}a^{(k+1)}_s+a^{(k+1)}_{u(j)}+a^{(k+1)}_{u(k+1)}=\sum_{s\in I}a^{(k)}_i+a^{(k)}_{u(j)}+\sum_{s\in I}a^{(k+1)}_s+a^{(k+1)}_{u(i)}+a^{(k+1)}_{u(k+1)}.
$$
This implies 
\begin{equation}
\label{eq: edge vectors aligned}
a^{(k)}_{u(i)}-a^{(k)}_{u(j)}=a^{(k')}_{u(i)}-a^{(k')}_{u(j)}.
\end{equation} 
and the result follows.
\end{proof}

\begin{corollary}
\label{cor: edge vector}
Assume that $u'=u(i\ j)$ covers $u$ and $u, u' \in [v,w]$. Then the edge vector between the corresponding vertices $X_u,X_{u'}$ of $\widehat{P}(v,w)$ equals 
$$
\overrightarrow{X_uX_{u'}}=(j-i)\left(a^{(i)}_{u(j)}-a^{(i)}_{u(i)}\right),
$$
and in particular each entry of this integral vector is divisible by $j-i$. 
\end{corollary}

\begin{remark}\label{rem:edges-in-Bruhat-interval-polytope}
    We note that each cover relation $u \lessdot u(i ~j) = u'$ also corresponds to a Bruhat interval polytope $P(u, u')$ which is a single edge, with vertices $x_u:=u \cdot (n-1, \dots, 0)$ and $x_u:=u' \cdot (n-1, \dots, 0)$. One can check that 
    \[\overrightarrow{x_{u} x_{u'}} = (j-i) (e_{u(j)} - e_{u(i)}).\]
\end{remark}

\begin{remark}
The product of positroid varieties $\Pi_k(v,w)$ is a toric variety with respect to the $(n-1)^2$-dimensional torus $T^{n-1}$, where each copy of the standard torus $T$ acts on the respective factor $\Pi_k(v,w)$. 

The Richardson variety $\overline{R}_{v,w}$ is cut out in  $\Pi_{k=1}^{n-1}\Pi_k(v,w)$ by the incidence Pl\"ucker relations \eqref{eq: incidence Plucker}. We can interpret $\TT_{v,w}$ as the maximal subtorus of  $T^{n-1}$ which makes the relations \eqref{eq: incidence Plucker} homogeneous. 

The tuple of matrices $(A^{(1)},\ldots,A^{(n-1)})$ from Lemma \ref{lem: transformation} describes the embedding $\TT_{v,w}\hookrightarrow T^{n-1}$. The condition \eqref{eq: edge vectors aligned} is one constraint describing the compatibility of $(A^{(1)},\ldots,A^{(n-1)})$ with the incidence Pl\"ucker relations, and it would be interesting to find a complete set of such constraints. See \cite[Section 7.3]{flagmatroid} for related constructions.

\end{remark}

Next, we consider a 2-dimensional face in the polytope $\widehat{P}(v,w)$. It is a quadrilateral with vertices $u,u_1,u_2,u_3$ which corresponds to a rank 2 subinterval $[u,u_3] \subset [v,w]$. Up to the symmetries $u_1\leftrightarrow u_2$ and $u\leftrightarrow u_3$, there are two types of such intervals and faces described by the following result.

\begin{lemma}
\label{lem: 2d face}
\begin{enumerate}[(a)]
    \item  Suppose $i,j,k,l$ are all distinct, and
$
u_1=u\cdot (i\ j),\  u_2=u\cdot (k\ l),\ u_3=u\cdot (i\ j)(k\ l).
$
Then 
\begin{equation}
\label{eq: parallelogram}
\overrightarrow{X_{u}X_{u_1}}=\overrightarrow{X_{u_2}X_{u_3}},\quad \overrightarrow{X_{u}X_{u_2}}=\overrightarrow{X_{u_1}X_{u_3}},
\end{equation}
so $u,u_1,u_2,u_3$ are vertices of a parallelogram.
\item Suppose $i<j<k$, and
$
u_1=u\cdot (i\ j),\  u_2=u\cdot (j\ k),\ u_3=u\cdot (i\ j\ k).
$
Then
\begin{equation}
\label{eq: trapezoid}
\overrightarrow{X_{u_1}X_{u_3}}=\frac{k-j}{j-i}\overrightarrow{X_uX_{u_1}}+\overrightarrow{X_uX_{u_2}},\quad \overrightarrow{X_{u_2}X_{u_3}}=\frac{k-i}{j-i}\overrightarrow{X_{u}X_{u_1}}.
\end{equation}
so $u,u_1,u_2,u_3$ are vertices of a trapezoid.
\end{enumerate}
\end{lemma}

\begin{proof}
Part (a) is clear from Corollary \ref{cor: edge vector}, let us prove (b).
Note that $u_3=u_1(j\ k)=u_2(i\ k).$ Now by Corollary \ref{cor: edge vector} we get
$$
\overrightarrow{X_uX_{u_1}}=(j-i)\left(a^{(i)}_{u(j)}-a^{(i)}_{u(i)}\right),\ 
\overrightarrow{X_uX_{u_2}}=(k-j)\left(a^{(j)}_{u(k)}-a^{(j)}_{u(j)}\right),\ 
\overrightarrow{X_{u_1}X_{u_3}}=(k-j)\left(a^{(j)}_{u(k)}-a^{(j)}_{u(i)}\right),
$$
and
$$
\overrightarrow{X_{u_2}X_{u_3}}=(k-i)\left(a^{(i)}_{u(j)}-a^{(i)}_{u(i)}\right)=(k-i)\left(a^{(j)}_{u(j)}-a^{(j)}_{u(i)}\right).
$$
The last equation follows from Lemma \ref{lem: edge vectors 1} since $i<j<k$. Now the equations \eqref{eq: trapezoid} are clear.
\end{proof}

\begin{remark}
    In the setting of Lemma~\ref{lem: 2d face}, the Bruhat interval polytope $P(u, u_3)$ is also a quadrilateral, with edges $\overrightarrow{x_{u} x_{u_1}}$,$\overrightarrow{x_{u} x_{u_2}}$, $\overrightarrow{x_{u_1} x_{u_3}}$, $\overrightarrow{x_{u_2} x_{u_3}}$ (using the notation of Remark~\ref{rem:edges-in-Bruhat-interval-polytope}). Using Remark~\ref{rem:edges-in-Bruhat-interval-polytope}, one can verify the edges of $P(u,u_3)$ satisfy precisely the relations of Lemma~\ref{lem: 2d face}. In other words, the edges in the face of $\widehat{P}(v,w)$ corresponding to $[u, u_3]$ satisfy the same relations as the edges in the Bruhat interval polytope $P(u, u_3)$. This is reminiscent of Dyer's vector space from \cite{Dyer97}. 

\end{remark}

\begin{corollary}
\label{cor: polytope from atoms}
The vertices 
\[\{X_v\} \cup \{X_u: v \lessdot u \leq w\}\]
uniquely determine the polytope $\widehat{P}(v,w)$. 
\end{corollary}

\begin{proof}
Given a two-dimensional face as in Lemma \ref{lem: 2d face}, the equations \eqref{eq: parallelogram} and \eqref{eq: trapezoid} uniquely determine  the vectors
$\overrightarrow{X_{u_1}X_{u_3}}$ and $\overrightarrow{X_{u_2}X_{u_3}}$ given the vectors $\overrightarrow{X_{u}X_{u_1}}$ and $\overrightarrow{X_{u}X_{u_2}}$. Therefore given the vertices $u,u_1,u_2$ one can uniquely reconstruct the vertex $u_3$ and proceed by induction. 
\end{proof}

\subsection{Moment polytopes using MR coordinates}

The previous results on the moment polytope $\widehat{P}_{v,w}$ and the $k$th summand polytope $\widehat{P}_k(v,w)$ hold regardless of the choice of multiplicative structure on $R_{v,w}$. In this subsection, we explicitly describe $\widehat{P}_k(v,w)$ when the coordinates on $R_{v,w}$ are the MR parameters. 

Fix a reduced expression $\bw$ for $w$, the corresponding 3D plabic graph $G_{v, \bw}$, and MR matrix $g$. We also define the \emph{trimmed graph} $G_{v, \bw}^k$ which is obtained from $G_{v, \bw}$ by deleting the edges which end at $k+1, \dots, n$ on the right. An \emph{NI path collection} in $G_{v,w}^k$ is a NI path collection from $k$ sources on the left to the $k$ sinks on the right. 

Recall from Lemma~\ref{lem:minors in MR-param} that for $I \in \CM_k(v,w)$, there is a unique NI path collection $P$ in $G_{v, \bw}^k$ which starts at $I$, and
\[\Delta_I(g) = \prod_{b_j \in P} t_j\]
where the product is over bridges in $P$. This immediately implies the following.

\begin{proposition} \label{prop:polytope from MR param}
    Using MR parameters as the coordinates on the torus $R_{v,w}$, the $k$th summand polytope $\widehat{P}^{MR}_k(v,w)$ is 
    \[\widehat{P}^{MR}_k(v,w)= \Conv\left(e_J: J \text{ is the set of bridges used in an NI path collection in } G_{v, \bw}^k\right) \subset \R^{J^\circ}.\]
\end{proposition}

This shows that in particular, $\widehat{P}^{MR}_k(v,w)$ is contained inside the unit cube $[0,1]^{\ell(w) - \ell(v)}$. The path collection from $v[k] \to [k]$ does not use any bridges, so the origin is always a vertex of $\widehat{P}^{MR}_k(v,w)$. 

\section{Examples} \label{sec:examples}
In this section, we give several examples of toric Richardson varieties $\overline{R}_{v,w}$, their moment polytopes 
$$\widehat{P}(v, w) = \widehat{P}_1(v, w)+ \cdots + \widehat{P}_{n-1}(v, w),$$ 
and the affine transformations from the positroid polytopes $P_k(v,w)$ to the constituent polytopes $\widehat{P}_k(v,w)$. The exact realization of the polytopes $\widehat{P}(v, w)$ and $\widehat{P}_k(v,w)$ in $\R^d$ depends on the choice of coordinates $x_1,\ldots,x_d$ on the torus $\TT_{v,w}$.

  Recall from Section \ref{subsec: polytopes} that for $I\in \CM_k$ we have $\Delta_I=x^{m_I}$ for some $m^I\in \Z^d$. The vertices of $\widehat{P}_k(v,w)$ are $\{m_I: I \in \CM_k(v,w)\}$ and by \eqref{eq: vertices} the vertices of $\widehat{P}(v,w)$ are given by the sums over nested sequences of subsets:
  $$
  m_{I_1}+m_{I_2}+\ldots+m_{I_{n-1}},\ I_1\subset I_2\subset \cdots \subset I_{n-1},\ |I_k|=k.
  $$
  
  For simplicity, we omit commas and parentheses and use the notation $m_I[\textcolor{red}{I}]$ to indicate the exponent vector and the subset labeling it. For example, if $m_{13} = (0,1,1,0)$, we write $0110[\textcolor{red}{13}]$. We use $\mathbf{0}$ for the zero vector in $\mathbb{Z}^{d}$.

We will often use the following useful fact, which follows from \cite[Corollary 4.10]{FZ}. 

\begin{proposition}
\label{prop: FZ}
A matrix $M$ represents a flag is in the opposite Schubert cell $X_v$ if and only if the following conditions hold for all $i<j$:
\begin{itemize}
\item[(a)] If $v(i)>v(j)$ then $\Delta_{v[i] \setminus v(i) \cup v(j)}=0$
\item[(b)] If $v(i)<v(j)$ then $\Delta_{v[i]}\neq 0$.
\end{itemize}
\end{proposition}


\subsection{First examples}

\begin{example}\label{ex: base case}
We revisit \cite[Example 5.17]{GKSS2}. 
Consider $w=4231$ and $v=1324\in S_4$. The interval $[v,w]$ is a $4$-dimensional hypercube, meaning it is isomorphic to the Boolean lattice of rank $4$, and $R_{v,w}$ is isomorphic to a $4$-dimensional torus $\TT_{v,w}$ with coordinates $x_1, x_2, x_3, x_4$. Indeed, one can check from Proposition \ref{prop: FZ} (or Lemma \ref{lem: defining eq for 5.18} below) the following matrix parametrizes
$R_{v,w}$:
 \begin{equation}\label{eq:braid matrix for w}
    M=\left(
\begin{matrix}
x_2x_3 & -x_3 & x_4 & -1 \\ x_2 & -1 & 0 & 0 \\ x_1 & 0 & -1 & 0 \\ 1 & 0 & 0 & 0
\end{matrix}
\right).
\end{equation}
One can verify using Sage and Theorem \ref{thm: polytope} that 

\begin{equation*}
\begin{aligned}
   & \widehat{P}_1(v,w)=\Conv(0110[\textcolor{red}{1}], 0100[\textcolor{red}{2}], 1000[\textcolor{red}{3}], 0000[\textcolor{red}{4}]), \\
   & \widehat{P}_2(v,w)=\Conv(1010[\textcolor{red}{13}],0010[\textcolor{red}{14}],1000[\textcolor{red}{23}],0000[\textcolor{red}{24}]),\\
   & \widehat{P}_3(v,w)=\Conv(1001[\textcolor{red}{123}],0001[\textcolor{red}{124}],0010[\textcolor{red}{134}],0000[\textcolor{red}{234}])
\end{aligned}
\end{equation*}





Here we omit commas and parentheses for brevity and use the notation $m_I[\textcolor{red}{I}]$ as mentioned above, so $1010[\textcolor{red}{13}]$ means that the exponent vector $m_{13}$ is equal to $(1,0,1,0)$. The polytopes $\widehat{P}_1(v,w)$ and $\widehat{P}_3(v,w)$ are 3-dimensional simplices and $\widehat{P}_2(v,w)$ is a square. We have
$$
\widehat{P}(v,w)=\widehat{P}_1(v,w)+\widehat{P}_2(v,w)+\widehat{P}_3(v,w),
$$
and one can check that $\widehat{P}(v,w)$ is indeed combinatorially equivalent to the 4-dimensional cube.  

The equations \eqref{eq: edge vectors aligned}  comparing the edge vectors in different polytopes  
have the form
$$
a^{(1)}_1-a^{(1)}_2=a^{(2)}_1-a^{(2)}_2=a^{(3)}_1-a^{(3)}_2,\quad a^{(1)}_3-a^{(1)}_4=a^{(2)}_3-a^{(2)}_4=a^{(3)}_3-a^{(3)}_4.
$$
In  other words, each edge vector of the square $\widehat{P}_2(v,w)$ must also be an edge vector of both $\widehat{P}_1(v,w)$ and $\widehat{P}_3(v,w)$. The remaining edge vectors of the latter two polytopes are unconstrained. 


The MR parametrization leads to a slightly different polytope $\widehat{P}^{\MaRi}(v,w)$, which is affinely equivalent to $\widehat{P}(v,w)$. 
The MR matrix for $\bw=s_1s_2s_3s_2s_1$ and $v=1324$ is 
 \begin{equation}\label{eq:braid matrix for w}
    g=\left(
\begin{matrix}
1 & 0 & 0 & 0 \\ t_1 & 0 & -1 & 0\\ t_5 & 1 & -t_2 & 0 \\ t_3t_5 & t_3 & 0 & 1
\end{matrix}
\right).
\end{equation}
Using Proposition~\ref{prop:polytope from MR param}, as illustrated in Figure~\ref{fig:path-collections}, (or by direct calculation), we obtain

\begin{equation*}
    \begin{aligned}
&\widehat{P}_1^{\MaRi}(v,w)=\Conv(0000[\textcolor{red}{1}],1000[\textcolor{red}{2}],0001[\textcolor{red}{3}],0011[\textcolor{red}{4}]), \\
&\widehat{P}_2^{\MaRi}(v,w)=\Conv(0000[\textcolor{red}{13}],0010[\textcolor{red}{14}],1000[\textcolor{red}{23}],1010[\textcolor{red}{24}]),\\
&\widehat{P}_3^{\MaRi}(v,w)=\Conv(0000[\textcolor{red}{123}],0010[\textcolor{red}{124}], 0110[\textcolor{red}{134}],1110[\textcolor{red}{234}]).
    \end{aligned}
\end{equation*}

\begin{figure}[ht]
    \centering
    \begin{subfigure}{0.48\textwidth} 
        \centering
        \begin{tikzpicture}[scale=1.3, line width=0.8pt] 
            \foreach \y/\label in {0/1, 0.5/2, 1.0/3, 1.5/4} {
                \draw [-{Stealth[scale=0.6]}] (0,\y) -- (1,\y);
                \draw (1,\y) -- (1.5,\y);
                \node[left, scale=0.8] at (0,\y) {$\label$};
            }
            \draw (1.5,0) -- (2.5,0); \draw (2.0,0.5) -- (2.5,0.5); \draw (2.0,1) -- (2.5,1); \draw (1.5,1.5) -- (2.5,1.5);
            \draw [-{Stealth}](0.25, 0.5) -- (0.25, 0.0);
            \node[right, purple, scale=0.8] at (0.25, 0.25) {$t_1$};
            \filldraw[fill=white] (0.25, 0.5) circle (2pt);
            \draw[blue] (0,1)--(1.5,1.0);
            \filldraw[fill=black] (0.75, 0.5) circle (2pt);
            \draw[-{Stealth}] (0.75, 1.0) -- (0.75, 0.5);
            \node[right, purple, scale=0.8] at (0.75, 0.75) {$t_2$};
            \filldraw[fill=white] (0.75, 1.0) circle (2pt);
            \filldraw[fill=black] (1.25, 1.0) circle (2pt);
            \draw[-{Stealth}] (1.25, 1.5) -- (1.25, 1.0);
            \node[right, purple, scale=0.8] at (1.25, 1.25) {$t_3$};
            \filldraw[fill=white] (1.25, 1.5) circle (2pt);
            \draw[red] (0,0)--(2.5,0);
            \filldraw[fill=black] (0.25, 0.0) circle (2pt);
            \filldraw[fill=black] (2.25, 0.0) circle (2pt);
            \draw[blue] (2.0,0.5)--(2.5,0.5);
            \draw[-{Stealth}] (2.25, 0.5) -- (2.25, 0.0);
            \node[right, purple, scale=0.8] at (2.25, 0.25) {$t_5$};
            \filldraw[fill=white] (2.25, 0.5) circle (2pt);
            \draw[thick,blue] (1.5,1.0) to [out=0,in=180] (2.0,0.5);
            \draw[color=white, line width=4] (1.5,0.5) to [out=0,in=180] (2.0,1.0);
            \draw[thick] (1.5,0.5) to [out=0,in=180] (2.0,1.0);
        \end{tikzpicture}
        \caption{$P:\{1,3\} \to [2]$, $\wt(P)=1$}
    \end{subfigure}%
    \hfill
    \begin{subfigure}{0.48\textwidth}
        \centering
        \begin{tikzpicture}[scale=1.3, line width=0.8pt]
            \foreach \y/\label in {0/1, 0.5/2, 1.0/3, 1.5/4} {
                \draw [-{Stealth[scale=0.6]}] (0,\y) -- (1,\y);
                \draw (1,\y) -- (1.5,\y);
                \node[left, scale=0.8] at (0,\y) {$\label$};
            }
            \draw (1.5,0) -- (2.5,0); \draw (2.0,0.5) -- (2.5,0.5); \draw (2.0,1) -- (2.5,1); \draw (1.5,1.5) -- (2.5,1.5);
            \filldraw[fill=black] (0.25, 0.0) circle (2pt);
            \draw [-{Stealth}](0.25, 0.5) -- (0.25, 0.0);
            \node[right, purple, scale=0.8] at (0.25, 0.25) {$t_1$};
            \filldraw[fill=white] (0.25, 0.5) circle (2pt);
            \draw[blue] (0,1.5)--(1.25,1.5);
            \filldraw[fill=black] (0.75, 0.5) circle (2pt);
            \draw[-{Stealth}] (0.75, 1.0) -- (0.75, 0.5);
            \node[right, purple, scale=0.8] at (0.75, 0.75) {$t_2$};
            \filldraw[fill=white] (0.75, 1.0) circle (2pt);
            \draw[-{Stealth},blue] (1.25, 1.5) -- (1.25, 1.0);
            \draw[blue](1.25,1.0)--(1.5,1.0);
            \filldraw[fill=black] (1.25, 1.0) circle (2pt);
            \node[right, purple, scale=0.8] at (1.25, 1.25) {$t_3$};
            \filldraw[fill=white] (1.25, 1.5) circle (2pt);
            \draw[red] (0,0)--(2.5,0);
            \filldraw[fill=black] (0.25, 0.0) circle (2pt);
            \filldraw[fill=black] (2.25, 0.0) circle (2pt);
            \draw[blue] (2.0,0.5)--(2.5,0.5);
            \draw[-{Stealth}] (2.25, 0.5) -- (2.25, 0.0);
            \node[right, purple, scale=0.8] at (2.25, 0.25) {$t_5$};
            \filldraw[fill=white] (2.25, 0.5) circle (2pt);
            \draw[thick,blue] (1.5,1.0) to [out=0,in=180] (2.0,0.5);
            \draw[color=white, line width=4] (1.5,0.5) to [out=0,in=180] (2.0,1.0);
            \draw[thick] (1.5,0.5) to [out=0,in=180] (2.0,1.0);
        \end{tikzpicture}
        \caption{$P:\{1,4\}\to [2], \text{wt}(P)=t_3$}
    \end{subfigure}

    \vspace{0.8cm} 

    \begin{subfigure}{0.48\textwidth}
        \centering
        \begin{tikzpicture}[scale=1.3, line width=0.8pt]
            \foreach \y/\label in {0/1, 0.5/2, 1.0/3, 1.5/4} {
                \draw [-{Stealth[scale=0.6]}] (0,\y) -- (1,\y);
                \draw (1,\y) -- (1.5,\y);
                \node[left, scale=0.8] at (0,\y) {$\label$};
            }
            \draw (1.5,0) -- (2.5,0); \draw (2.0,0.5) -- (2.5,0.5); \draw (2.0,1) -- (2.5,1); \draw (1.5,1.5) -- (2.5,1.5);
              \draw[red] (0.25,0)--(2.5,0);
               \draw [-{Stealth},red](0.25, 0.5) -- (0.25, 0.0);
            \filldraw[fill=black] (0.25, 0.0) circle (2pt);
           
            \node[right, purple, scale=0.8] at (0.25, 0.25) {$t_1$};
            \draw[red] (0,0.5)--(0.25,0.5);
            \filldraw[fill=white] (0.25, 0.5) circle (2pt);
            \draw[blue] (0,1)--(1.5,1.0);
            \filldraw[fill=black] (0.75, 0.5) circle (2pt);
            \draw[-{Stealth}] (0.75, 1.0) -- (0.75, 0.5);
            \node[right, purple, scale=0.8] at (0.75, 0.75) {$t_2$};
            \filldraw[fill=white] (0.75, 1.0) circle (2pt);
           
            \filldraw[fill=black] (1.25, 1.0) circle (2pt);
            \draw[-{Stealth}] (1.25, 1.5) -- (1.25, 1.0);
            \node[right, purple, scale=0.8] at (1.25, 1.25) {$t_3$};
            \filldraw[fill=white] (1.25, 1.5) circle (2pt);
           
            \filldraw[fill=black] (2.25, 0.0) circle (2pt);
            \draw[blue] (2.0,0.5)--(2.5,0.5);
            \draw[-{Stealth}] (2.25, 0.5) -- (2.25, 0.0);
            \node[right, purple, scale=0.8] at (2.25, 0.25) {$t_5$};
            \filldraw[fill=white] (2.25, 0.5) circle (2pt);
            \draw[thick,blue] (1.5,1.0) to [out=0,in=180] (2.0,0.5);
            \draw[color=white, line width=4] (1.5,0.5) to [out=0,in=180] (2.0,1.0);
            \draw[thick] (1.5,0.5) to [out=0,in=180] (2.0,1.0);
        \end{tikzpicture}
        \caption{$P:\{2,3\}\to [2], \text{wt}(P)=t_1$}
    \end{subfigure}%
    \hfill
    \begin{subfigure}{0.48\textwidth}
        \centering
        \begin{tikzpicture}[scale=1.3, line width=0.8pt]
            \foreach \y/\label in {0/1, 0.5/2, 1.0/3, 1.5/4} {
                \draw [-{Stealth[scale=0.6]}] (0,\y) -- (1,\y);
                \draw (1,\y) -- (1.5,\y);
                \node[left, scale=0.8] at (0,\y) {$\label$};
            }
            \draw (1.5,0) -- (2.5,0); \draw (2.0,0.5) -- (2.5,0.5); \draw (2.0,1) -- (2.5,1); \draw (1.5,1.5) -- (2.5,1.5);
            \draw[red] (0,0.5)--(0.25,0.5);
                 \draw[red] (0.25,0)--(2.5,0);
                  \draw [-{Stealth},red](0.25, 0.5) -- (0.25, 0.0);
            \filldraw[fill=black] (0.25, 0.0) circle (2pt);
           
            \node[right, purple, scale=0.8] at (0.25, 0.25) {$t_1$};
            
            \filldraw[fill=white] (0.25, 0.5) circle (2pt);
            \draw[blue] (0,1.5)--(1.25,1.5);
            \filldraw[fill=black] (0.75, 0.5) circle (2pt);
            \draw[-{Stealth}] (0.75, 1.0) -- (0.75, 0.5);
            \node[right, purple, scale=0.8] at (0.75, 0.75) {$t_2$};
            \filldraw[fill=white] (0.75, 1.0) circle (2pt);
            \draw[-{Stealth},blue] (1.25, 1.5) -- (1.25, 1.0);
            \draw[blue](1.25,1.0)--(1.5,1.0);
            \filldraw[fill=black] (1.25, 1.0) circle (2pt);
            \node[right, purple, scale=0.8] at (1.25, 1.25) {$t_3$};
            \filldraw[fill=white] (1.25, 1.5) circle (2pt);
       
            \filldraw[fill=black] (2.25, 0.0) circle (2pt);
            \draw[blue] (2.0,0.5)--(2.5,0.5);
            \draw[-{Stealth}] (2.25, 0.5) -- (2.25, 0.0);
            \node[right, purple, scale=0.8] at (2.25, 0.25) {$t_5$};
            \filldraw[fill=white] (2.25, 0.5) circle (2pt);
            \draw[thick,blue] (1.5,1.0) to [out=0,in=180] (2.0,0.5);
            \draw[color=white, line width=4] (1.5,0.5) to [out=0,in=180] (2.0,1.0);
            \draw[thick] (1.5,0.5) to [out=0,in=180] (2.0,1.0);
        \end{tikzpicture}
        \caption{$P:\{2,4\}\to [2], \text{wt}(P)=t_1t_3$}
    \end{subfigure}
    \caption{NI path collections on a 3D plabic graph and their weights.}
    \label{fig:path-collections}
\end{figure}
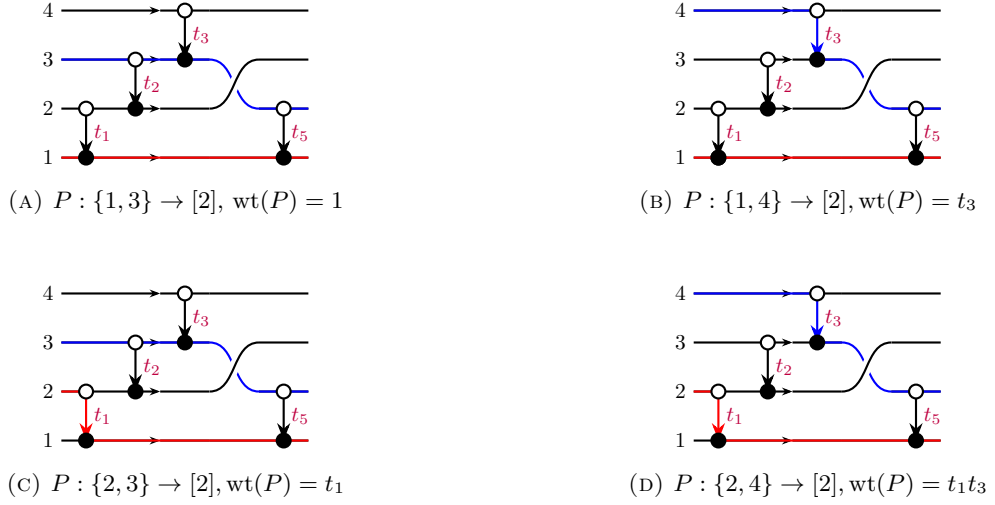


and again
$$
\widehat{P}^{\MaRi}(v,w)=\widehat{P}_1^{\MaRi}(v,w)+\widehat{P}_2^{\MaRi}(v,w)+\widehat{P}_3^{\MaRi}(v,w).
$$

Lastly, we give affine transformations $\alpha \mapsto A^{(k)} \alpha + b^{(k)}$ that maps the positroid polytope ${P}_k(v,w)$ to $\widehat{P}_k(v,w)$.  

\begin{equation*}
    A^{(1)}:  =\begin{pmatrix}
    0 & 0 & 1 & 0\\
    1 & 1 & 0 & 0\\
    1 & 0 & 0 & 0\\
    0 & 0 & 0 & 0
    \end{pmatrix}
,\ b^{(1)}=0\end{equation*}

\begin{equation*}
    A^{(2)}: =\begin{pmatrix}
    0 & 0 & 1 & 0\\
    0 & 0 & 0 & 0\\
    1 & 0 & 0 & 0\\
    0 & 0 & 0 & 0
    \end{pmatrix}
    ,\ b^{(2)}=0
\end{equation*}

\begin{equation*}
   A^{(3)}: =\begin{pmatrix}
    0 & 0 & 0 & -1\\
    0 & 0 & 0 & 0\\
    0 & -1 & 0 & 0\\
    0 & 0 & -1 & -1
    \end{pmatrix}
    ,\ b^{(3)}=\left(
    \begin{matrix}
    1 \\ 0\\ 1\\ 2
    \end{matrix}
    \right)
\end{equation*}

\end{example}

\begin{example}\label{ex:4-crown}($4$-crown)
We consider an example of an expected toric Richardson $\overline{R}_{v,w}$ where $[v,w]$ is a $4$-crown.
Consider $w=4231$ and $v=2143\in S_4$. Then one can check by using Proposition \ref{prop: FZ} that 
$R_{v,w}$ is parametrized by the matrix
 \begin{equation}\label{eq:braid matrix for w}
    M=\left(
\begin{matrix}
0 & -x_3 & x_1^{-1}x_2x_3 & -1 \\ x_2 & -1 & 0 & 0 \\ x_1 & 0 & -1 & 0 \\ 1 & 0 & 0 & 0
\end{matrix}
\right),
\end{equation}
where $x_1,x_2,x_3$ are coordinates in the torus $\TT$.
We get

\begin{equation*}
    \begin{aligned}
&\widehat{P}_1(v,w)=\Conv(010[\textcolor{red}{2}],100[\textcolor{red}{3}],000[\textcolor{red}{4}]), \\
&\widehat{P}_2(v,w)=\Conv(011[\textcolor{red}{12}],101[\textcolor{red}{13}],001[\textcolor{red}{14}],100[\textcolor{red}{23}],000[\textcolor{red}{24}]),\\
&\widehat{P}_3(v,w)=\Conv(-111[\textcolor{red}{124}],001[\textcolor{red}{134}],000[\textcolor{red}{234}]).
    \end{aligned}
\end{equation*}
Note that $\widehat{P}_1(v,w)$ and $\widehat{P}_3(v,w)$ are triangles, while $\widehat{P}_2(v,w)$ is a square pyramid. The polytope $\widehat{P}(v,w)$ is depicted in Figure \ref{fig:4crown polytope}. 
One can see that all faces are quadrilaterals. The cone in the dual fan corresponding to a 4-valent vertex is not simplicial, so the toric variety is not smooth.\\

Alternatively, one can compute the polytope from
the \emph{MR} matrix 
 \begin{equation*}\label{eq:braid matrix for w}
    g=\left(
\begin{matrix}
0 & -1 & 0 & 0 \\ 1 & -t_1 & 0 & 0\\ t_2 & 0 & 0 & -1 \\ t_4 & 0 & 1 & 0
\end{matrix}
\right).
\end{equation*}
Using Lemma \ref{lem:minors in MR-param}, we obtain 

\begin{equation*}
    \begin{aligned}
&\widehat{P}_1^{\MaRi}(v,w)=\Conv(000[\textcolor{red}{2}],010[\textcolor{red}{3}],001[\textcolor{red}{4}]), \\
&\widehat{P}_2^{\MaRi}(v,w)=\Conv(000[\textcolor{red}{12}],010[\textcolor{red}{13}],001[\textcolor{red}{14}],110[\textcolor{red}{23}],101[\textcolor{red}{24}]),\\
&\widehat{P}_3^{\MaRi}(v,w)=\Conv(000[\textcolor{red}{124}],010[\textcolor{red}{134}],110[\textcolor{red}{234}]).
    \end{aligned}
\end{equation*}

\begin{figure}
    \centering
    \includegraphics[width=0.5\linewidth]{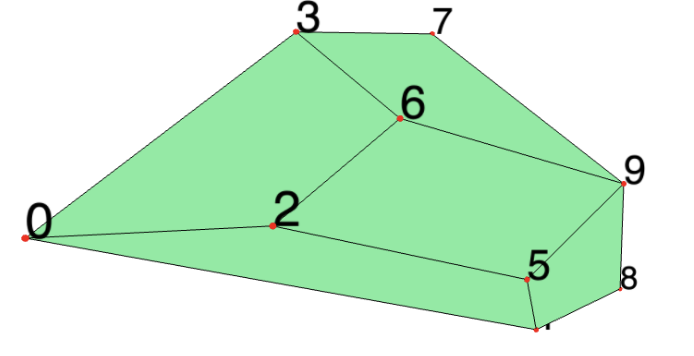}
    \caption{The polytope $\widehat{P}(v,w)$ from $4$-crown}
    \label{fig:4crown polytope}
\end{figure}

\end{example}


\begin{remark}
    Though the examples presented in detail here are rank-symmetric, this is not true in general for intervals with  no 2-crowns. For example, for $v=32154$ and $w=53241$ in $S_5$, the interval $[v,w]$ has rank sizes $(1,5,9,6,1)$.
\end{remark}

\subsection{An infinite family of examples}\label{subsec-infinite-family}
For even $n$, we consider $w=(1\ n)$ and $v=2143 \dots n (n-1)=s_2s_4\cdots s_{n-2}$ as in \cite[Remark 5.18]{GKSS2}. 
The Schubert cell $X^w$ is parametrized by the matrix below.
$$
M=\left(
\begin{matrix}
z_{n-1} & -z_n & z_{n+1} & \cdots & z_{2n-3} & -1\\
z_{n-2} & -1 & 0 & \cdots & 0  & 0\\
\vdots &  & \ddots &  & \vdots & \vdots \\
z_2 & 0 & \cdots & -1 & 0 & 0\\
z_1 & 0 & \cdots & 0 & -1 & 0\\
1 & 0 & \cdots & 0 &  0 & 0\\ 
\end{matrix}
\right)
$$
\begin{lemma}\label{lem: defining eq for 5.18}
The open Richardson variety $R_{v,w}$ is the subset of $X^w$ where all $z_i$ are invertible, $z_{n-1}=z_{n-2}z_{n}$ and 
\begin{equation}
\label{eq: z binomial}
z_{i}z_{2n-2-i}=z_{i+1}z_{2n-3-i}\ \mathrm{for\ all\ even\ } i< n-2.
\end{equation}
\end{lemma}

\begin{proof}
We use Proposition \ref{prop: FZ} to characterize the matrices in the opposite Schubert cell $X_v$. 
First, we consider the case $j=i+1$. For $i$ odd, we have $v(i)<v(i+1)$ and $v([i])=[i]$, so $\Delta_{[i]}\neq 0$. For $i$ even, we have $v(i)>v(i+1)$ and $v(\{1,\ldots,i-1,i+1\})=[i]$, so $\Delta_{[i]}=0$.

We have a recurrence relation for principal minors of $M$:
\begin{equation}
\label{eq: principal minors recursion}
\Delta_{[i]}=-\Delta_{[i-1]}+z_{n-i}z_{n+i-2}
\end{equation}
For $i=1$ we get $\Delta_{[1]}=z_{n-1}\neq 0$. For $i=2$ we get $\Delta_{[2]}=-z_{n-1}+z_{n}z_{n-2}=0$, so $z_{n-1}=z_{n-2}z_{n}$ and $z_{n},z_{n-2}\neq 0$. Now we proceed by induction in $i$. If $i>1$ is odd then $\Delta_{[i-1]}=0$ hence by \eqref{eq: principal minors recursion} we get $\Delta_{[i]}=z_{n-i}z_{n+i-2}\neq 0$, so both $z_{n-i}$ and $z_{n+i-2}$ are nonzero. Furthermore, 
$$
\Delta_{[i+1]}=-\Delta_{[i]}+z_{n-i-1}z_{n+i-1}=-z_{n-i}z_{n+i-2}+z_{n-i-1}z_{n+i-1}=0.
$$
This is equivalent to the equation \eqref{eq: z binomial} and also implies that $z_{n-i-1},z_{n+i-1}\neq 0$. So all $z_i$ are nonzero. 

It remains to check that there are no additional relations if $j>i+1$. Indeed, in this case $v(i)<v(j)$ and $v([i])=[i]$ for $i$ odd or $i=2n$ and $v(i)=\{1,
\ldots,i-1,i+1\}$ for $i<2n$ even. In the former case we already checked $\Delta_{[i]}\neq 0$, and in the latter we can use the identity
$$
\Delta_{\{1,
\ldots,i-1,i+1\}}=z_{n-i-1}z_{n+i-2}\neq 0.
$$
\end{proof}
Lemma \ref{lem: defining eq for 5.18} implies that one can choose the coordinates $x_i$ on $\TT_{v,w}$ as follows:
$$
x_{i}=\begin{cases}
z_i & \mathrm{if}\ i\le n-2, \\
z_n & \mathrm{if}\ i=n-1, \\
z_{2i+1-n}  & \mathrm{if}\ n\le i\le \frac{3n}{2}-2. 
\end{cases}
$$

Then all remaining $z_i$ can be written as monomials in $x_j$:
$$
z_{n-1}=x_{n-1}x_{n-2},\ z_{2i-n}=\frac{x_{3n-2i-1}x_{i-1}}{x_{3n-2i-2}}\ \mathrm{if}\ n+1\le i\le \frac{3n}{2}-2.
$$

We can write all flag minors as monomials in the $x_i$. For convenience, we define $x_0:=1$.

\begin{lemma}
\label{lem: 5.1.8 matroid}
The flag minors of a point in $R_{v,w}= \TT_{v,w}$ are as follows.
\begin{itemize}
\item[(1)] For $k$ odd, we have
$$
\Delta_{[k]}=\begin{cases}
x_{n-2}x_{n-1} & \mathrm{if}\ k=1, \\
x_{n-k}x_{n+\frac{k-3}{2}}  & \mathrm{if}\ k\neq 1.\\
\end{cases}
$$

\item[(2)]  For any $k$ and $1\le i\le k<j$ we have 

\begin{equation*}
\Delta_{[k]\setminus i\cup j}=
    \begin{cases}
    
  x_{n-j} & \mathrm{if}\ i=1\\
  x_{n-j}x_{n-1}  & \mathrm{if}\ i=2\\
  x_{n-j}x_{n+\frac{i-3}{2}} & \mathrm{odd}\ i\ge 3\\ 
x_{n-\frac{i+4}{2}}^{-1}x_{n-j}x_{n+\frac{i-4}{2}}x_{n-\frac{i+2}{2}}  & \mathrm{even}\ i\ge 4\\        
    \end{cases}
\end{equation*}

\item[(3)] $\Delta_{[n]}=1$
\end{itemize}
All other flag minors $\Delta_{I}$ vanish on $R_{v,w}$.
\end{lemma}

\begin{corollary}
The vectors $a_i^{(k)}$ can be chosen as follows:
\begin{enumerate}
\item For $k=1$, $b^{(1)}=0$ and
$$
a_i^{(1)}=\begin{cases}
e_{n-2}+e_{n-1} & \mathrm{if}\ i=1, \\
e_{n-i}  & \mathrm{otherwise}.\\
\end{cases}
$$
    
\item For $k=2$, $b^{(2)}=0$ and
$$
a_i^{(2)}=\begin{cases}
0  & \mathrm{if}\ i=2, \\
e_{n-i}  & \mathrm{otherwise}.\\
\end{cases}
$$

\item For $k\ge 3$ odd, $b^{(k)}=e_{n-k}+e_{n+\frac{k-3}{2}}+B_k$ and
 $$
a_i^{(k)}=\begin{cases}
0 & \mathrm{if}\ i=1, \\
-e_{n-1} & \mathrm{if}\ i=2, \\
-e_{n+\frac{i-3}{2}}   & \mathrm{for}\ \mathrm{ odd}\ i\le k, \\
-e_{n+\frac{i-4}{2}}-e_{n-\frac{i+2}{2}}+e_{n-\frac{i+4}{2}} & \mathrm{for}\ \mathrm{ even}\ \ i\le k-1, \\

e_{n-i}-e_{n-k}-e_{n+\frac{k-3}{2}} & \mathrm{otherwise},\\
\end{cases}
$$

\item For $k\ge 4$ even, $b^{(k)}=0$ and
$$
a_i^{(k)}=\begin{cases}
0 & \mathrm{if}\ i=1, \\
-e_{n-1} & \mathrm{if}\ i=2, \\
-e_{n+\frac{i-3}{2}}   & \mathrm{for}\ \mathrm{ odd}\ i\le k-1, \\
-e_{n+\frac{i-4}{2}}-e_{n-\frac{i+2}{2}}+e_{n-\frac{i+4}{2}} & \mathrm{for}\ \mathrm{ even}\ \ i\le k, \\

e_{n-i}+B_k & \mathrm{otherwise}.\\
\end{cases}
$$

\end{enumerate}
In (3) and (4), we use
\begin{equation*}
    \begin{aligned}
& B_k:=e_{n-1}+\sum_{l=1}^{\left\lfloor\frac{k-1}{2}\right\rfloor}e_{n+l-1}+\sum_{l=2}^{\left\lfloor\frac{k}{2}\right\rfloor}(e_{n+l-2}+e_{n-l-1}-e_{n-l-2})
    \end{aligned}
\end{equation*}
and we set $e_0:=0$.

\end{corollary}

\begin{proof}
    The proof essentially follows from Lemma \ref{lem: 5.1.8 matroid}. Part (1) is clear. To prove (2), we can check that the $a_j^{(k)}$ above is indeed a solution of the system of linear equations $a_1^{(2)}+a_j^{(2)}=e_{n-j}+e_{n-1}$ and $a_2^{(2)}+a_j^{(2)}=e_{n-j}$ for $3\le j\le n$. 
    
    For (3), we first compute $\sum_{i=1}^{k}a_{i}^{(k)}=-B_k$. Then 
    $$
    A^{(k)}e_{[k]}+b^{(k)}=-B_k+e_{n-k}+e_{n-\frac{k-3}{2}}+B_k=e_{n-k}+e_{n-\frac{k-3}{2}}
    $$
    while for $i\le k,j\ge k$ we get
    $$
    A^{(k)}e_{[k]\setminus i\cup j}+b^{(k)}=\left(A^{(k)}e_{[k]}+b^{(k)}\right)-a_i^{(k)}+a_j^{(k)}=
    $$
    $$
    e_{n-k}+e_{n-\frac{k-3}{2}}-a_i^{(k)}+a_j^{(k)}=e_{n-j}-a_{i}^{(k)}.
    $$
    This agrees with $m_{[k]\setminus i\cup j}$ by Lemma \ref{lem: 5.1.8 matroid}. The proof  for   (4) is similar.
\end{proof}


Alternatively, for the choice of the reduced word $\bw=s_1s_2\cdots s_{n-1}\cdots s_2s_1$ and $v=s_2s_4\cdots s_{n-2}$, the nonzero flag minors of MR matrix can be explicitly computed using the following two lemmas.

\begin{lemma}\label{lem: basecase MR for minors}
Consider the choice of $v$ and $\bw$ as above. Then following flag minors of the associated \emph{MR} matrix are computed as follows:

\begin{enumerate}[(a)]
  \item The minors $\Delta_{v[k]}=1$.
    \item For $i\ge 2$,
    \begin{equation*} 
\Delta_{i}=
    \begin{cases}
 t_{2q-1}\cdot(\prod_{j=0}^{q-2}t_{2n-3-2j}) & \mathrm{for}\ i=2q,\\
 \prod_{j=0}^{q-2}t_{2n-3-2j} & \mathrm{for}\ i=2q-1.\\
    \end{cases}
\end{equation*}
\end{enumerate}
\end{lemma}

\begin{proof}
 Recall from Lemma~\ref{lem:minors in MR-param} that there is an unique NI path collection $P: I \to [k]$ for any $I \in \CM_k(v,w)$, and thus we only need to specify the NI path collection for each $I$. For (a), if $k$ is odd, the desired path collection consists of paths that start at $1,2,\cdots, k$ and do not use any bridges. This path collection ends at $[k]$ because $v[k]=[k]$ and is non-intersecting. Analogously, for $k$ even, the desired path collection consists of paths that start at $[k-1]\cup \{k+1\}$ that do not use any bridges. For (b), by inspection the desired NI path collection consists of the single path which starts at $i$ and uses the bridges with the weights in (b). See Figure~\ref{fig:infinite-family-graph}.

\end{proof}

\begin{remark}
It is clear that $\Delta_{2k}=t_{2k-1}\Delta_{2k-1}$.
\end{remark}

 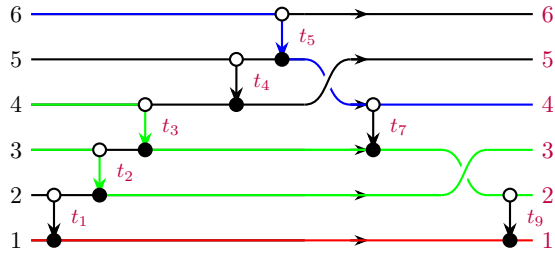
\begin{figure}[ht]
    \centering

\begin{tikzpicture}[scale=1.2, line width=0.8pt]
    \foreach \y/\label in {0/1, 0.5/2, 1.0/3, 1.5/4, 2.0/5, 2.5/6} {
        \draw (0,\y) -- (3,\y);
        \draw[-{Stealth[scale=0.8]}] (3.5, \y) -- (3.7, \y);
        \node[left, scale=0.9] at (0,\y) {$\label$};
        \node[right, scale=0.9,purple] at (5.5,\y) {$\label$};
    
    }
 \draw [red](0,0) -- (5.5,0); 
    
    \filldraw[fill=black] (0.25, 0) circle (2pt);
    \draw[-{Stealth}] (0.25, 0.5) -- (0.25, 0);
    \filldraw[fill=white] (0.25, 0.5) circle (2pt);
    \node[right, purple, scale=0.8] at (0.35, 0.25) {$t_1$};

    \draw[green]  (0,1)--(0.75,1);
    \draw[green]  (0.75,0.5)--(4.5,0.5);
    \draw[green,-{Stealth}] (0.75, 1) -- (0.75, 0.5);
    \filldraw[fill=white] (0.75, 1) circle (2pt);
    \filldraw[fill=black] (0.75, 0.5) circle (2pt);
    \node[right, purple, scale=0.8] at (0.85, 0.75) {$t_2$};

     \draw[green]  (0,1.5)--(1.25,1.5);
    \draw[green]  (1.25,1)--(4.5,1.0);
     \draw[-{Stealth},green] (1.25, 1.5) -- (1.25, 1);
    \filldraw[fill=black] (1.25, 1) circle (2pt);
   
    \filldraw[fill=white] (1.25, 1.5) circle (2pt);
    \node[right, purple, scale=0.8] at (1.35, 1.25) {$t_3$};


    
    \filldraw[fill=black] (2.25, 1.5) circle (2pt);
    \draw[-{Stealth}] (2.25, 2) -- (2.25, 1.5);
    \filldraw[fill=white] (2.25, 2) circle (2pt);
    \node[right, purple, scale=0.8] at (2.35, 1.75) {$t_{4}$};

         \draw[blue]  (0,2.5)--(2.75,2.5);
             \draw[blue]  (2.75,2)--(3,2);
             \draw[-{Stealth},blue] (2.75, 2.5) -- (2.75, 2);
    \filldraw[fill=black] (2.75, 2) circle (2pt);
    \filldraw[fill=white] (2.75, 2.5) circle (2pt);
    \node[right, purple, scale=0.8] at (2.85, 2.25) {$t_{5}$};


\draw (3,2.5)--(5.5,2.5);\draw (3.5,2) -- (5.5,2); \draw[blue] (3.5,1.5) -- (5.5,1.5);\draw[green] (5,0.5) -- (5.5,0.5);\draw[green] (5,1) -- (5.5,1);\draw[green] (3,0.5) -- (4.5,0.5);\draw[green] (3,1) -- (4.5,1);


     \draw[thick,blue] (3,2) to [out=0,in=180] (3.5,1.5);
    \draw[color=white, line width=4] (3,1.5) to [out=0,in=180] (3.5,2);
    \draw[thick] (3,1.5) to [out=0,in=180] (3.5,2.0);

\draw[blue] (3.5,1.5)--(3.75,1.5);
 \filldraw[fill=black] (3.75, 1) circle (2pt);
    \draw[-{Stealth}] (3.75, 1.5) -- (3.75,1);
    \filldraw[fill=white] (3.75, 1.5) circle (2pt);
    \node[right, purple, scale=0.8] at (3.85, 1.25) {$t_{7}$};


     \draw[thick,green] (4.5,1) to [out=0,in=180] (5,0.5);
    \draw[color=white, line width=4] (4.5,0.5) to [out=0,in=180] (5,1);
    \draw[thick,green] (4.5,0.5) to [out=0,in=180] (5,1.0);

    \filldraw[fill=black] (5.25, 0) circle (2pt);
    \draw[-{Stealth}] (5.25, 0.5) -- (5.25,0);
    \filldraw[fill=white] (5.25, 0.5) circle (2pt);
    \node[right, purple, scale=0.8] at (5.35, 0.25) {$t_{9}$};

\end{tikzpicture}
\caption{3D plabic graph for $n=6$ for the family of examples in Section~\ref{subsec-infinite-family}. The colored paths give examples of the types of paths in the proof of Lemma~\ref{lem: infinite series MR minor} for $\Delta_{[4] \setminus 2 \cup 6} $: the type 1 path is in \textcolor{red}{red}, type 2 paths are in \textcolor{green}{green}, and the type 3 path is in \textcolor{blue}{blue}.}
    \label{fig:infinite-family-graph}
\end{figure}

\begin{lemma}\label{lem: infinite series MR minor}
Assume the choice of $v$ and $\bw$ as above. Then

\begin{equation*}
    \Delta_{[k]\setminus i\cup j}=\begin{cases}
      (\prod_{q=i+1}^{k}t_{q-1})\cdot \Delta_j/{\Delta_k}   & \mathrm{for}\ k\ \mathrm{odd}, \\
      (\prod_{q=i+1}^{k}t_{q-1})\cdot {\Delta_j}/{\Delta_{k+1}}   & \mathrm{for}\ k\ \mathrm{even}.\\
    \end{cases}
\end{equation*}

\end{lemma}

\begin{proof}
The desired NI path collection consists of 
 \begin{itemize}
     \item Type 1 paths: the paths from the proof of Lemma \ref{lem: basecase MR for minors} (a) which start at $1, \dots i-1$. 
     \item Type 2 paths: paths starting at $l$ for $i+1\le l\le k$ which use only the bridge $t_{l-1}$.
     \item Type 3 path: this starts the same as the path from the proof of Lemma~\ref{lem: basecase MR for minors} (b) which starts at $j$, but does not use the bridges $t_a$ for $a>2n-k$. This path ends at $v(k)$, which is $k-1$ if $k$ is odd and is $k$ if $k$ is even.
 \end{itemize}

 To see that these paths are in fact non-intersecting, note that the type 1 paths consist of wires in a wiring diagram for $v$, and, after taking the bridge $t_l$, the type 2 paths also consist of wires in the wiring diagram. By inspection, (see Figure \ref{fig:infinite-family-graph} for an example) these wires are distinct since the type 2 paths are above the type 1 paths. 
It is also clear from Figure \ref{fig:infinite-family-graph} that the type $3$ path lies above the other paths so does not intersect those.



\end{proof}

\begin{corollary}
The vectors $a_m^{(k)}$ can be chosen as follows:

\begin{enumerate}
\item For $k=1$, $b^{(1)}=0$ and
$$
a_m^{(1)}=\begin{cases}
0 & \mathrm{if}\ m=1, \\
e_{2k-1}+\sum_{l=0}^{k-2}e_{2n-3-2j} & \mathrm{if}\ m=2k.\\
\sum_{j=0}^{k-2}e_{2n-3-2j} & \mathrm{if}\ m=2k-1.\\
\end{cases}
$$
\item For $k\ge 2$ even, $b^{(k)}=0$ and

$$
a_m^{(k)}=\begin{cases}
-\sum_{l=m}^{k-1}e_l & \mathrm{if}\ 1\le m\le k-1,\\
0 & \mathrm{if}\ m=k\\
a_m^{(1)}-a_{k+1}^{(1)}+\sum_{l=1}^{k-1}le_l & \mathrm{if}\ m\ge k+1\\
\end{cases}
$$

\item For $k\ge 3$ odd, $b^{(k)}=-\sum_{j=1}^{k}(k-j)e_{j}$ and
$$
a_m^{(k)}=\begin{cases}
\sum_{l=1}^{i}e_{l-1} & \mathrm{if}\ 1\le m\le k, \\
a_{m}^{(1)}-a_{k}^{(1)}+\sum_{l=1}^{k}e_{l-1} & \mathrm{if}\ m\ge k+1,                 \\
\end{cases}
$$
where we set $e_0=0$.
\end{enumerate}

\end{corollary}

\begin{proof}
The part (1) directly follows from (a) and (c) of Lemma \ref{lem: basecase MR for minors}. The part (2) and (3) essentially follows from Lemma \ref{lem: infinite series MR minor}. To be specific, 
for $k\ge 2$ even, we can check that $a_m^{(k)}$ above is indeed a solution of the system of linear equations $
A^{(k)}e_{[k]\setminus i \cup j}=
(\sum_{l=1}^{k}a_l^{(k)})-a_i^{(k)}+a_j^{(k)}=\sum_{j=i+1}^{k}e_{j-1}+a_{j}^{(1)}-a_{k+1}^{(1)}$ for $1\le i\le k<j$. Similarly, for $k\ge 3$ odd, we can check that $a_m^{(k)}$ above is indeed a solution of the system of linear equations $A^{(k)}e_{[k]\setminus i \cup j}+b^{(k)}=(\sum_{l=1}^{k}a_l^{(k)})-a_i^{(k)}+a_j^{(k)}-\sum_{l=1}^{k}(k-l)e_{l}=\sum_{l=i+1}^{k}e_{l-1}+a_{j}^{(1)}-a_{k}^{(1)}$ and $A^{(k)}e_{[k]}+b^{(k)}=0$.

\end{proof}


We can also describe the constituent positroid polytopes geometrically.

\begin{lemma}
\label{lem: polytope even}~
\begin{enumerate}[(a)]
    \item If $k$ is even, the positroid polytope $P_k(v,w)$ (and hence $\widehat{P}_k(v,w)$) is affinely equivalent to the product of a $(k-1)$-dimensional simplex and $(n-k-1)$-dimensional simplex. 
    \item If $k$ is odd, the positroid polytope $P_k(v,w)$ (and hence $\widehat{P}_k(v,w)$) is affinely equivalent to the cone over the product of a $(k-1)$-dimensional simplex and $(n-k-1)$-dimensional simplex.  
\end{enumerate}
\end{lemma}

\begin{proof}
a) For $k$ even, the nonzero minors are $I=[k]\setminus i\cup j$ for $1\le i\le k$ and $k<j\le n$. We have 
\begin{align*}
P_k(v,w)=&\Conv\left(e_{[k]}-e_i+e_j:\ 1\le i\le k,\ k<j\le n\right)
\\
=&e_{[k]}+\Conv\left(-e_i+e_j:\ 1\le i\le k,\ k<j\le n\right)
\\
=&e_{[k]}+\Conv\left(-e_i:\ 1\le i\le k\right)+\Conv\left(e_j:k<j\le n\right)
\end{align*}
where in the last step we used the identity of Minkowski sums $\Conv(A+B)=\Conv(A)+\Conv(B)$. Now $e_{[k]}$ is an overall shift, $\Conv\left(-e_i:\ 1\le i\le k\right)$ is a $(k-1)$-dimensional simplex, and $\Conv\left(e_j:k<j\le n\right)$ is a $(n-k-1)$-dimensional simplex which lives in an orthogonal space, so the result follows.

b) For $k$ odd, the nonzero minors are $I=[k]$ and $I=[k]\setminus i\cup j$ for $1\le i\le k$ and $k<j\le n$. As in Lemma \ref{lem: polytope even} we can write 
$$
P_k(v,w)=e_{[k]}+\Conv\left(\mathbf{0}, \Conv\left(-e_i:\ 1\le i\le k\right)+\Conv\left(e_j:k<j\le n\right)\right).
$$
The sum $\Conv\left(-e_i:\ 1\le i\le k\right)+\Conv\left(e_j:k<j\le n\right)$ is contained in the codimension 2 hyperplane in $\R^n$ where the sum of the first $k$ coordinates equals $-1$ and the sum of the last $(n-k)$ coordinates equals $1$.
Since $\mathbf{0}$ is not contained in this hyperplane, we get the cone over $\Conv\left(-e_i:\ 1\le i\le k\right)+\Conv\left(e_j:k<j\le n\right)$.
\end{proof}





 Now we completely describe the plabic graphs $G_k(v,w)$.

\begin{lemma}\label{lem: plabic graph for 5.1.8}
Suppose that we are given with $k$ and $n$ as above. Then $G_k(v,w)$ follows the following dichotomy:

\begin{itemize}
    \item [(a)] For $k$ odd, $G_k(v,w)$ is depicted in Figure \ref{fig: plabic graph for 5.1.8} (left).
    \item [(b)]For $k$ even, $G_k(v,w)$ is depicted in Figure \ref{fig: plabic graph for 5.1.8} (right).
\end{itemize}

\end{lemma}


\begin{proof}
It is sufficient to check that the set of all perfect orientations on $G_k(v,w)$ yields the correct positroid $\CM_k(v,w)$. 

For $k$ odd, suppose that the interior edge of the left graph is directed from the black vertex to the white vertex. Then the boundary vertices $1,\ldots,k$ are sources and $k+1,\ldots,n$ are sinks. Such an orientation $\mathcal{O}$ corresponds to $I_{\mathcal{O}}=[k]$. Otherwise the interior edge is directed from white to black vertex, there is a unique sink $i$ among $1,\ldots,k$ and a unique source $j$ among $k+1,\ldots,n$. Here $i\le k$ and $j>k$ are arbitrary, and we recover $I_{\mathcal{O}}=[k]\setminus i\cup j$. The proof for $k$ even is similar.
\end{proof}

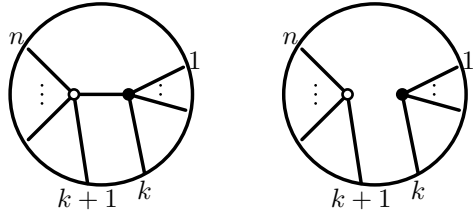
\begin{figure}

\begin{tikzpicture}[scale=0.6, line cap=round, line join=round]

\begin{scope}[shift={(0,0)}]

\draw[line width=1.3pt] (0,0) circle (2);

\coordinate (W) at (-0.6,0);
\coordinate (B) at (0.6,0);

\draw[line width=1.3pt] (W) -- (B);

\draw[line width=1.3pt] (W) -- (-1.6,1.0);
\draw[line width=1.3pt] (W) -- (-1.6,-1.0);

\draw[line width=1.3pt] (B) -- (1.8,0.6);
\draw[line width=1.3pt] (B) -- (1.85,-0.35);

\draw[line width=1.3pt] (W) -- (-0.3,-1.97);

\draw[line width=1.3pt] (B) -- (0.95,-1.75);

\fill (-1.3,0.2) circle (0.03);
\fill (-1.3,0.0) circle (0.03);
\fill (-1.3,-0.2) circle (0.03);

\fill (1.3,0.2) circle (0.03);
\fill (1.3,0.0) circle (0.03);
\fill (1.3,-0.2) circle (0.03);

\draw[line width=1pt, fill=white] (W) circle (0.11);
\draw[line width=1pt, fill=black] (B) circle (0.11);

\node at (-1.85,1.25) {$n$};
\node at (2.05,0.75) {$1$};

\node at (-0.3,-2.3) {$k+1$};
\node at (0.95,-2.1) {$k$};

\end{scope}


\begin{scope}[shift={(6.0,0)}]

\draw[line width=1.3pt] (0,0) circle (2);

\coordinate (W) at (-0.6,0);
\coordinate (B) at (0.6,0);


\draw[line width=1.3pt] (W) -- (-1.6,1.0);
\draw[line width=1.3pt] (W) -- (-1.6,-1.0);

\draw[line width=1.3pt] (B) -- (1.8,0.6);
\draw[line width=1.3pt] (B) -- (1.85,-0.35);

\draw[line width=1.3pt] (W) -- (-0.3,-1.97);

\draw[line width=1.3pt] (B) -- (0.95,-1.75);

\fill (-1.3,0.2) circle (0.03);
\fill (-1.3,0.0) circle (0.03);
\fill (-1.3,-0.2) circle (0.03);

\fill (1.3,0.2) circle (0.03);
\fill (1.3,0.0) circle (0.03);
\fill (1.3,-0.2) circle (0.03);

\draw[line width=1pt, fill=white] (W) circle (0.11);
\draw[line width=1pt, fill=black] (B) circle (0.11);

\node at (-1.85,1.25) {$n$};
\node at (2.05,0.75) {$1$};

\node at (-0.3,-2.3) {$k+1$};
\node at (0.95,-2.1) {$k$};

\end{scope}

\end{tikzpicture}

\caption{For $n$ even, set $w=(1~n)$ and $v=s_2 s_4 \dots s_{n-2}$. The plabic graphs $G_k(v,w)$ for $k$ odd (left) and $k$ even (right).}
\label{fig: plabic graph for 5.1.8}

\end{figure}


\subsection{An infinite family from big hypercubes}\label{subsec-hypercube}

In the recent paper \cite{hypercube}, the authors defined two permutations $v_n, w_n \in S_{2^n}$ (called $x_n$ and $y_n$ in that work) and showed that the poset $[v_n,w_n]$ is a hypercube of dimension $n\cdot2^{n-1}$. In our context, this means that $\overline{R}_{v_n,w_n}$ is a toric variety of dimension $n\cdot2^{n-1}$, and the corresponding moment polytope $\widehat{P}(v_n,w_n)$ is combinatorially equivalent to a cube. 
We first recall the definition of $v_n$ and $w_n$ from \cite[Definition 2.4]{hypercube}.

\begin{definition} \label{defn: xn and yn}
We inductively define the permutations $v_n\in S_{2^n}$ in a one-line notation as follows: \begin{enumerate}
    \item $v_1=(1,2)\in S_2$,
    \item $v_{n+1} =\left(v_{n}(1),v_n(1)+2^n,v_n(2),v_n(2)+2^n,\cdots,v_n(2^n),v_n(2^n)+2^n\right)\in S_{2^{n+1}}$.
    \item $w_{n+1} = (v_{n+1}({2^{n+1}}),v_{n+1}({2^{n+1}-1}), \dots, v_{n+1}(2), v_{n+1}(1) ).$
\end{enumerate}

\end{definition}

The authors of \cite{hypercube} give an explicit description of the interval $[v_n, w_n]$, which we now recall. A \emph{basic $j$-interval} is an interval $[c\cdot 2^j+1,(c+1)2^{j}]\subset [1,2^n]$. The basic $j$-intervals have size $2^j$ and subdivide $[1,2^n]$ into $2^{n-j}$ equal parts.

\begin{definition}
A permutation $u\in S_{2^n}$ is called dyadically well-distributed if for any basic $j$-interval $S$ and any basic $(n-j)$-interval $T$ there exists a unique $a\in S$ such that $u(a)\in T$. 
\end{definition}

The main theorem of \cite{hypercube} then states that $u\in [v_n,w_n]$ if and only if $u$ is dyadically well-distributed. We can use this result to describe the positroids $\CM_k(n):=\CM_k(v_n,w_n)$.

\begin{proposition}
\label{prop: hypercube matroid}
We have $I=u[k]$ for some dyadically well-distributed permutation $u$ (equivalently, $I\in \CM_k(n)$)  if and only if for any $0\le j\le n$ and any basic $(n-j)$-interval $T$ one has
\begin{equation}
\label{eq: hypercube matroid}
\left\lfloor\frac{k}{2^j}\right\rfloor \le |I\cap T|\le \left\lceil\frac{k}{2^j}\right\rceil.
\end{equation}
In particular, if $2^j$ divides $k$ then for any  basic $(n-j)$-interval $T$ we get $|I\cap T|=\frac{k}{2^j}$.
\end{proposition}

\begin{proof}
Let us prove that \eqref{eq: hypercube matroid} is a necessary condition. Assume first that $2^j$ divides $k$, then $[k]$ can be decomposed into $\frac{k}{2^j}$ basic $j$-intervals. The permutation $u$ maps exactly one element in each of them into $T$, so $|u[k]\cap T|=\frac{k}{2^j}$.

More generally, define $k'=2^j\left\lfloor\frac{k}{2^j}\right\rfloor$ and $k''=2^j\left\lceil\frac{k}{2^j}\right\rceil$, then $k'\le k\le k''$ and 
$$
|u[k']\cap T|\le |u[k]\cap T|\le |u[k'']\cap T| 
$$
which implies \eqref{eq: hypercube matroid}.

Conversely, suppose we are given a subset $I$ satisfying \eqref{eq: hypercube matroid}. Let us prove that we can remove one element $i\in I$ such that $I\setminus i$ also satisfies \eqref{eq: hypercube matroid}. 
Indeed, suppose $k=2^l\cdot k'$ where $k'$ is odd. Each basic $(n-l)$-interval  contains exactly $k'$ elements of $I$. Let us pick a basic $n-l$ interval and denote by $T_l$.  We can split $T_l$ into two halves, one of them contains $\lfloor \frac{k'}{2}\rfloor$ elements of $I$ and the other $\lceil \frac{k'}{2}\rceil$ elements of $I$, let us call the latter $T_{l+1}$. 

More generally, assume that for $j\ge l$ and some $(n-j)$-interval $T_j$ we have $|I\cap T_{j}|>\left\lfloor \frac{k}{2^j}\right\rfloor$. Then 
$$
2\left\lfloor \frac{k}{2^{j+1}}\right\rfloor\le \left\lfloor \frac{k}{2^j}\right\rfloor<|I\cap T_{j}|
$$
so $I$ intersects one of the halves of $T_j$ by more than $\left\lfloor \frac{k}{2^{j+1}}\right\rfloor$ elements. We can call this half $T_{j+1}$ and proceed by induction. This yields a decreasing sequence of basic intervals $T_{l}\supset T_{l+1}\supset T_{l+1}\supset \cdots$ which all intersect in one element $i$.  By construction, $I\setminus i$ satisfies \eqref{eq: hypercube matroid} for $k-1$ instead of $k$.

By a similar argument, we can add an element to $I$ preserving the conditions \eqref{eq: hypercube matroid}. By iterating these procedures, we construct a nested sequence of subsets $I_1\subset \ldots \subset I_{n-1}$ such that $I_k=I$ and  $I_s$ satisfy the condition \eqref{eq: hypercube matroid} for all $s$. This sequence corresponds to a permutation $u$, and one can check that $u$ is indeed dyadically well-distributed.

\end{proof}


We will use Proposition~\ref{prop: hypercube matroid} to give an inductive description of the plabic graphs $G_k(n):=G_k(v_n, w_n)$. Examples of these graphs are in Figures~\ref{fig:n16-hypercube} and \ref{fig:n32-hypercube}. Recall that in Theorem \ref{thm: positroid torus}, we proved that the plabic graph $G_k(n)$ is a forest. 

\begin{definition}\label{defn: layer}

If $G$ is a tree plabic graph, the \emph{distance} of a vertex $v\in G$ from the boundary is the length of the shortest path from a boundary vertex to $a$.
\end{definition}

Just as the definition of $v_n, w_n$ is inductive, the easiest description of $G_k(n)$ is also inductive.

\begin{lemma}\label{lem: hypercube base case}
The graph $G_1(n)$ is a $2^n$-star graph with a white central vertex, meaning there is one internal vertex, which is white and adjacent to all $2^n$ boundary vertices.
\end{lemma}

\begin{proof}
One can easily prove by induction using Definition \ref{defn: xn and yn} that $v_n(1)=1$ and $v_n(2^n)=2^{n}$, so $w_n(1)=2^n$ and $w_n(2^n)=1$.
Using Proposition \ref{prop: FZ}, this implies that all $1\times 1$ flag minors $\Delta_{\{i\}}$ are nonzero. The graph described in the proposition has, for any $i$, a perfect orientation with source set $\{i\}$, so is the plabic graph for $\CM_1(n)$.
\end{proof}

\begin{figure}
    \centering
    \includegraphics[width=0.6\linewidth]{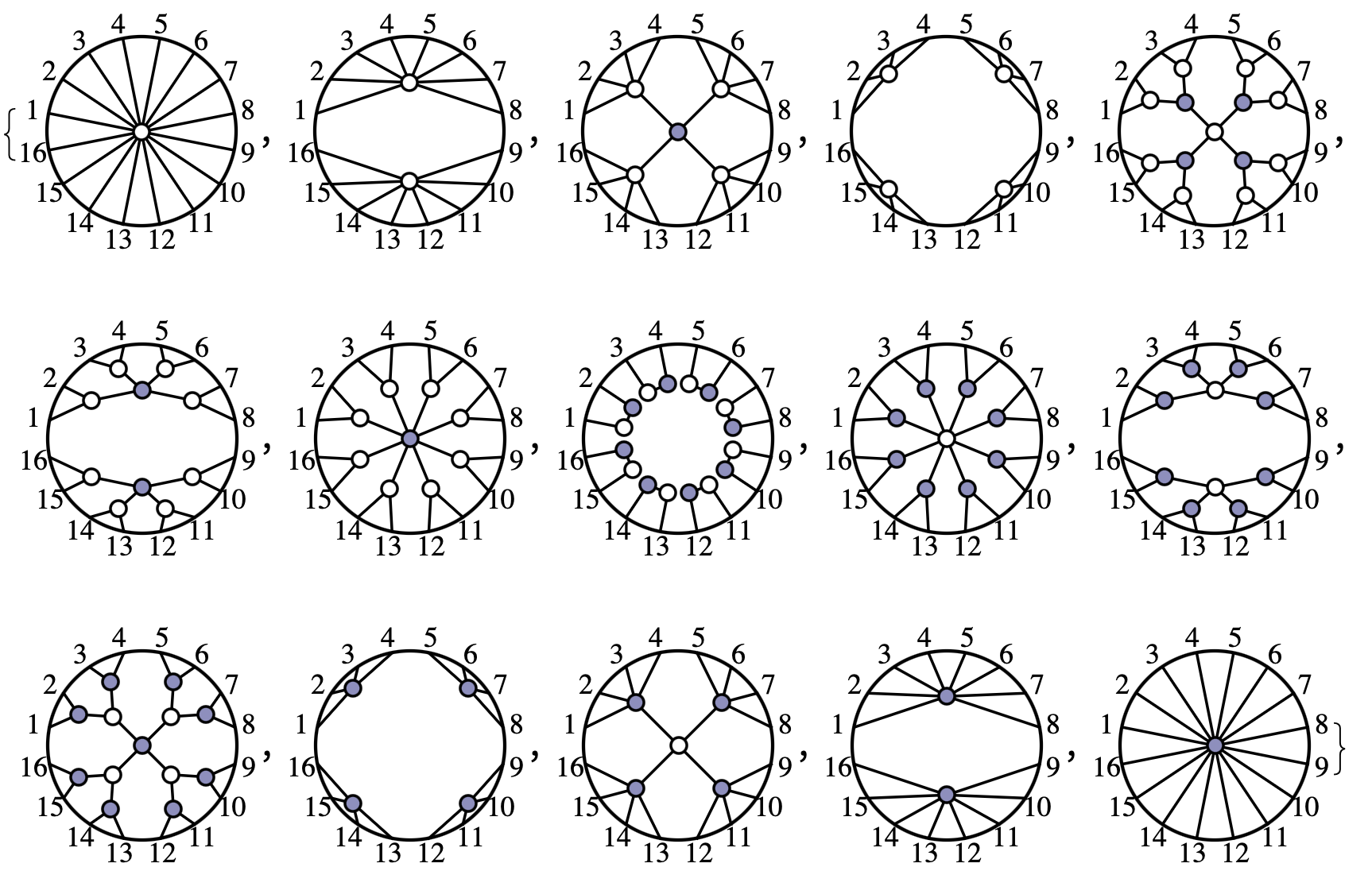}
    \caption{Graphs $G_k(4)$ for constituent positroids}
    \label{fig:n16-hypercube}
\end{figure}

Next, we recursively describe the graphs $G_k(n)$. For convenience, we temporarily allow plabic graphs $G$ with boundary vertices labeled by a totally ordered set $I$, and say $G$ is a plabic graph on $I$.


\begin{theorem}\label{thm: hypercube}
Assume $k=2^l\cdot k'$ where $k'$ is odd. This plabic graph on $[2^n]$ is obtained as follows.
\begin{enumerate}
\item  If $l>0$, then $G_k(n)$ is the disjoint union of $2^l$ copies of $G_{k'}(n-l)$, one on each basic $(n-l)$-interval.

\item If $l=0$ so $k=k'$, then $2^{L-1}<k<2^L$ for some $L$.

\begin{enumerate}
    \item If $L=0$, then $k=1$ and $G_1(n)$ is a star graph with a white internal vertex.
    \item If $L=n$, then $G_k(n)$ is obtained from $G_{n-k}(n)$ by swapping the colors of all internal vertices.
    \item If $0<L< n$, then consider the graph $G_k(L)$ on $0', \dots, ({2^L}-1)'$. The graph $G_k(n)$ is obtained from $G_k(L)$ by making all boundary vertices of $G_k(L)$ white and, for each $c$, adding $2^{n-L}$ legs from the vertex $c'$ to the basic $(n-L)$-interval $[c \cdot 2^{n-L}+1, (c+1)2^{n-L}]$.
\end{enumerate}

\end{enumerate}
\end{theorem}

\begin{proof}
(2a) is Lemma~\ref{lem: hypercube base case}. 
We describe the effect of (1), (2b), (2c) on perfect orientations. We verify the corresponding statements hold for the positroids $\CM_k(n)=\CM_k(v_n,w_n)$ in a series of lemmas below.

In (1), choosing a perfect orientation of $G = G_1 \sqcup \dots \sqcup G_r$ is equivalent to choosing a perfect orientation on each connected component $G_i$ independently. This corresponds to Lemma \ref{lem: disconnected} below.

In (2b), if $G^{op}$ is obtained from $G$ by switching the colors of every internal vertex, then reversing the direction of every edge gives a bijection between perfect orientations of $G$ and $G^{op}$. The effect of this bijection on source sets $I_{\mathcal{O}}$ is taking complements. This corresponds to Lemma \ref{lem: complement} below.

In (2c), every perfect orientation of $G_k(n)$ restricts to a perfect orientation of $G_k(L)$. Each perfect orientation of $G_k(L)$ can be extended to many different perfect orientations of $G_k(n)$. Indeed, fix a perfect orientation on $G_k(L)$. Given a distance 1 white vertex $v$, we consider the edge $e$ connecting $v$ to a distance 2 vertex. If $e$ is oriented towards $v$, then we direct all $2^{n-L}$ legs of $v$ towards the boundary and these boundary vertices do not contribute to $I_{\mathcal{O}}$. If $e$ is oriented away from $v$ then we may choose any of the $2^{n-L}$ legs of $v$ to orient towards $v$. This boundary vertex contributes to $I_{\mathcal{O}}$.
This corresponds to Lemma \ref{lem: layer 1} below.
\end{proof}

\begin{remark}
    The graph $G_{2^l}(n)$ has $2^l$ connected components where each component is a $2^{n-l}$-star graph with white central vertex. In Figures \ref{fig:n16-hypercube} and \ref{fig:n32-hypercube}, the components of the graphs $G_{2^{n-1}}$ have one interior white vertex and one interior black vertex, both of degree 2. Such a component is equivalent to a $2$-star with a white central vertex.
\end{remark}

\begin{example}
Consider the graph $G_{10}(5)$ as an example, see Figure \ref{fig:n32-hypercube}. First, $10=2^1\cdot 5$, so the graph $G_{10}(5)$ consists of two copies of $G_5(4)$ by (1), compare with Figure \ref{fig:n16-hypercube}. 

Next, we would like to describe $G_5(4)$. We have $2^2<5<2^3$ for $k=5$, so $L=3<4=n$. By (2c), $G_5(4)$ is obtained from $G_5(3)$ by making all boundary vertices white and adding two legs to each boundary vertex. For $G_5(3)$, $L=n=3$ so by (2b), $G_5(3)$ is $G_3(3)$ with colors swapped. 
The graph $G_3(3)$ is obtained from $G_3(2)$ by attaching two legs to each boundary vertex. Finally, the graph $G_3(2)$ is obtained from $G_1(2)$ by swapping colors, so is a $4$-star with black center vertex.
One can verify all these steps using Figure \ref{fig:n16-hypercube} and Figure \ref{fig:n32-hypercube}.
\end{example}

\begin{figure}
    \centering
    \includegraphics[width=0.6\linewidth]{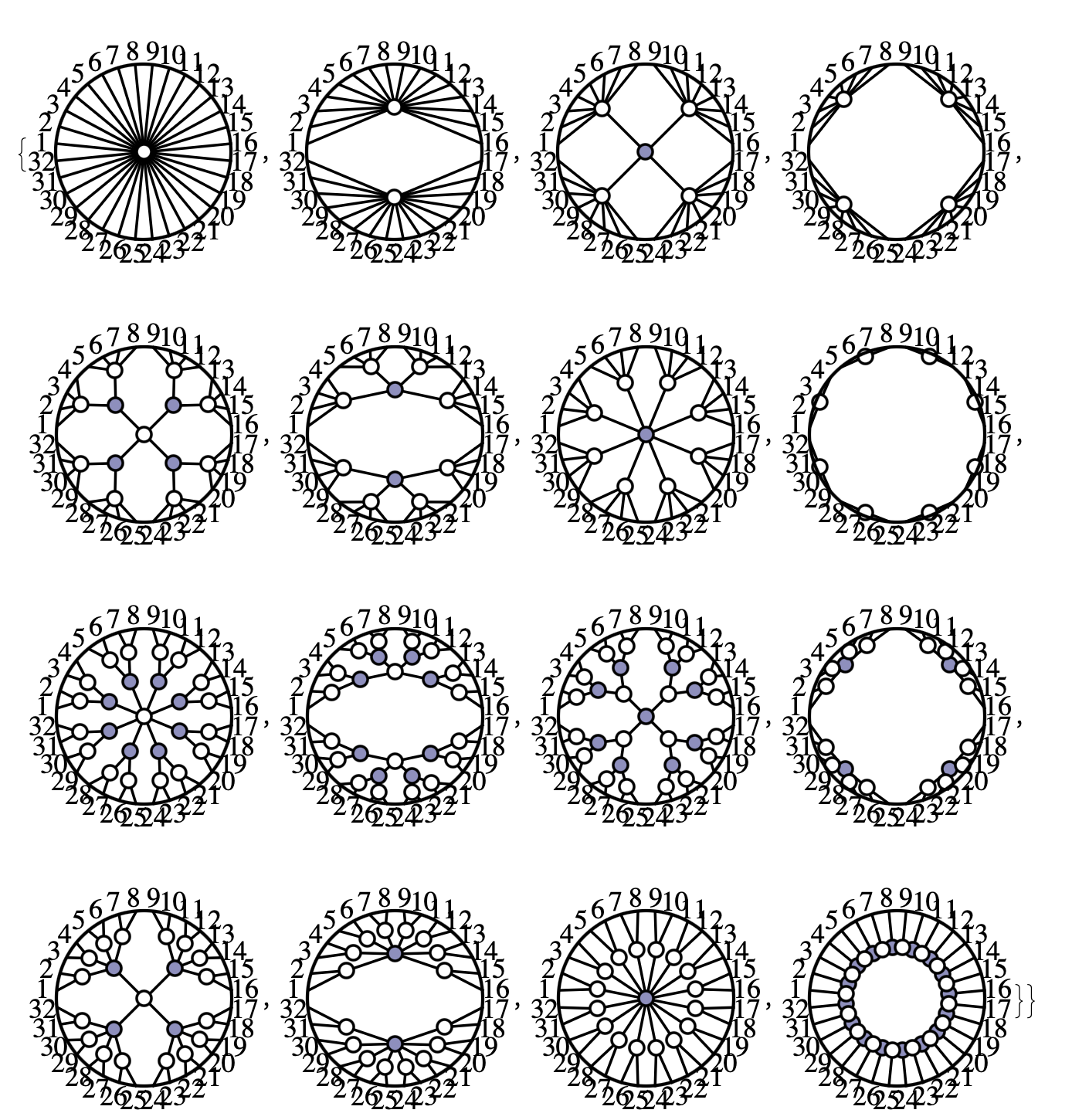}
    \caption{Graphs $G_k(5)$ for $k\le 16$. For $k\ge 16$, the graph $G_k(5)$ is the same as $G_{32-k}(5)$ up to switching the colors of vertices.}
    \label{fig:n32-hypercube}
\end{figure}

\begin{lemma}
\label{lem: disconnected}
Assume $k=2^l\cdot k'$ where $k'$ is odd. Then $I\in \CM_k(n)$ if and only if for each basic $(n-l)$-interval $T$ one has $|I\cap T|=k'$ and $I\cap T\in \CM_{k'}(n-l)$.
\end{lemma}

\begin{proof}
The condition $|I\cap T|=k'$ is necessary by \eqref{eq: hypercube matroid} for $j=l$. If $j<l$ then $2^j$ divides $k$ and  each basic $(n-j)$-interval $T'$ must contain $2^{l-j}k'$ elements of $I$. On the other hand, $T'$ can be decomposed into $2^{l-j}$ basic $(n-l)$-intervals and each of them contains $k'$ elements of $I$, and \eqref{eq: hypercube matroid} follows.

Finally, we need to analyze the case $j>l$. In this case we can decompose each basic $(n-l)$-interval $T$ into $2^{j-l}$ basic $(n-j)$-intervals $T'$, and 
$$
\left\lfloor\frac{k'}{2^{j-l}}\right\rfloor =\left\lfloor\frac{k}{2^j}\right\rfloor \le |I\cap T|\le \left\lceil\frac{k}{2^j}\right\rceil=\left\lceil\frac{k'}{2^{j-l}}\right\rceil
$$
so $I\cap T\in \CM_{k'}(n-l)$.
\end{proof}

\begin{lemma}
\label{lem: complement}
We have $I\in \CM_k(n)$ if and only if $[2^n]\setminus I=\overline{I}\in \CM_{2^n-k}(n)$. 
\end{lemma}

\begin{proof}
For each basic $(n-j)$-interval $T$ we have $|\overline{I}\cap T|=2^{n-j}-|I\cap T|$. On the other hand,
$$
\left\lfloor\frac{2^n-k}{2^{j}}\right\rfloor =2^{n-j}-\left\lceil\frac{k}{2^j}\right\rceil,\ \left\lceil\frac{2^n-k}{2^{j}}\right\rceil= 2^{n-j}-\left\lfloor\frac{k}{2^j}\right\rfloor
$$
and the result follows.
\end{proof}

\begin{lemma}
\label{lem: layer 1}
Assume $k\le 2^L$ with $L<n$. Then the subsets $I\in \CM_k(n)$ can be characterized as follows:
\begin{itemize}
\item[(a)] For each basic $(n-L)$-interval $T$, we have $|I\cap T|\le 1$.
\item[(b)] Let us identify the set of basic $(n-L)$-intervals with $[2^L]$ and define a subset $\widetilde{I}\subset [2^L]$ describing the $(n-L)$-intervals $T$ such that $I\cap T\neq \emptyset$.
Then have $I\in \CM_k(n)$ if and only if $\widetilde{I}\in \CM_k(L)$. 
\end{itemize}
\end{lemma}

\begin{proof}
Since $\frac{k}{2^L}\le 1$, the condition (a) is necessary by  \eqref{eq: hypercube matroid} with $j=L$. We need to check the inequalities \eqref{eq: hypercube matroid} for all other $j$.

If $j>L$ then $\frac{k}{2^j}< 1$ and \eqref{eq: hypercube matroid} requires that $|I\cap T'|\le 1$ for all basic $(n-j)$-intervals $T'$. On the other hand, each $T'$ is contained in some basic $(n-L)$-interval $T$, and $|I\cap T'|\le |I\cap T|\le 1$. 

If $j<L$ then each basic $(n-j)$-interval $T''$ is the union of $2^{L-j}$ basic $(n-L)$-intervals. By arguing as in (b), we can associate to $T''$ a basic $(L-j)$-interval $\widetilde{T''}$ in $[2^L]$. This gives a bijection between $(n-j)$-intervals in $[2^n]$ and $(L-j)$-intervals in $[2^L]$. Now $|T''\cap I|=|\widetilde{T''}\cap \widetilde{I}|$, so the conditions \eqref{eq: hypercube matroid} for $I$ and $\widetilde{I}$ are equivalent. 
\end{proof}

We finish with one example of the polytopes $\widehat{P}_k(v_n,w_n):=\widehat{P}_k(n)$. 


\begin{example}\label{ex: hypercube n=8}
Consider $n=3$. We have two permutations $v_3=(1,5,3,7,2,6,4,8)$ and $w_3=(8,4,6,2,7,3,5,1)$ in $S_8$ in one-line notation. Then one may verify using Proposition \ref{prop: FZ} and Sage that $R_{v_3,w_3}$ is parametrized by the matrix 
$$
M=\left(
\begin{matrix}
x_1x_4x_8 & -x_4x_8 & x_6x_8 & -x_8 & x_9x_{B} & -x_{\B} & x_{\CC} & -1\\
x_1x_4 & -x_4 & x_{6} & -1 & 0  & 0 & 0 & 0\\
x_1x_5 & -x_5 & 0 & 0 & -x_9 & 1 & 0 & 0 \\
x_1 & -1 & 0 & 0 & 0 & 0 & 0 & 0\\
x_2x_7 & 0 & -x_7 & 0 & x_{\A} & 0 & -1 & 0\\
x_2 & 0 & -1 & 0 &  0 & 0 & 0 & 0\\ 
x_3 & 0 & 0 & 0 & -1 & 0 & 0 & 0\\
1 & 0 & 0 & 0 & 0 & 0 & 0 & 0\\
\end{matrix}
\right)
$$
and thus $R_{v_3,w_3}\simeq (\C^*)^{12}=\TT$ where coordinates on $\TT$ is $x_i$ for $1\le i\le 12$. We use the notation $\A=10$, $\B=11$, $\CC=12$.

From this data, we obtain $\widehat{P}_{k}(3)$ as follows:  

\begin{equation*}
    \begin{aligned}
        \widehat{P}_{1}(3)&=\Conv\left(e_{148}[\textcolor{red}{1}], e_{14}[\textcolor{red}{2}],e_{15}[\textcolor{red}{3}],e_{1}[\textcolor{red}{4}],e_{27}[\textcolor{red}{5}],e_{2}[\textcolor{red}{6}],e_{3}[\textcolor{red}{7}],\mathbf{0}[\textcolor{red}{8}]\right),\\
        \widehat{P}_2(3)&=\Conv(e_{2478}[\textcolor{red}{15}],e_{248}[\textcolor{red}{16}],e_{348}[\textcolor{red}{17}],e_{48}[\textcolor{red}{18}],e_{247}[\textcolor{red}{25}],e_{24}[\textcolor{red}{26}],e_{34}[\textcolor{red}{27}],e_{4}[\textcolor{red}{28}],\\
        &\ \ \ \ \ \ \ \ \ \
        e_{257}[\textcolor{red}{35}],e_{25}[\textcolor{red}{36}],e_{35}[\textcolor{red}{37}],e_{5}[\textcolor{red}{38}],e_{27}[\textcolor{red}{45}],e_{2}[\textcolor{red}{46}],e_{3}[\textcolor{red}{47}], \mathbf{0}[\textcolor{red}{48}]),\\
        \widehat{P}_3(3)&=\Conv(e_{25678}[\textcolor{red}{135}],e_{2568}[\textcolor{red}{136}],e_{3568}[\textcolor{red}{137}],e_{568}[\textcolor{red}{138}],
        e_{2678}[\textcolor{red}{145}],e_{268}[\textcolor{red}{146}],e_{368}[\textcolor{red}{147}],e_{68}[\textcolor{red}{148}],\\
        &\ \ \ \ \ \ \ \ \ \
        e_{2567}[\textcolor{red}{235}],e_{256}[\textcolor{red}{236}],e_{356}[\textcolor{red}{237}],e_{56}[\textcolor{red}{238}],
        e_{267}[\textcolor{red}{245}],e_{26}[\textcolor{red}{246}],e_{36}[\textcolor{red}{247}],e_{6}[\textcolor{red}{248}],\\
          &\ \ \ \ \ \ \ \ \ \ e_{3478}[\textcolor{red}{157}],e_{478}[\textcolor{red}{158}],e_{348}[\textcolor{red}{167}],e_{48}[\textcolor{red}{168}],e_{347}[\textcolor{red}{257}],e_{47}[\textcolor{red}{258}],e_{34}[\textcolor{red}{267}],e_{4}[\textcolor{red}{268}],\\
          &\ \ \ \ \ \ \ \ \ \ 
          e_{357}[\textcolor{red}{357}],
          e_{57}[\textcolor{red}{358}],e_{35}[\textcolor{red}{367}],e_{5}[\textcolor{red}{368}],e_{37}[\textcolor{red}{457}],e_{7}[\textcolor{red}{458}],e_{3}[\textcolor{red}{467}],\mathbf{0}[\textcolor{red}{468}]),\\
        \hat{P}_4(3)&=\Conv(e_{3578}[\textcolor{red}{1357}],e_{578}[\textcolor{red}{1358}],e_{358}[\textcolor{red}{1367}],e_{58}[\textcolor{red}{1368}], 
        e_{378}[\textcolor{red}{1457}],e_{78}[\textcolor{red}{1458}],e_{38}[\textcolor{red}{1467}],e_{8}[\textcolor{red}{1468}],\\
        &\ \ \ \ \ \ \ \ \ \ 
        e_{357}[\textcolor{red}{2357}],e_{57}[\textcolor{red}{2358}],e_{35}[\textcolor{red}{2367}],e_{5}[\textcolor{red}{2368}],
        e_{37}[\textcolor{red}{2457}],e_7[\textcolor{red}{2458}],e_3[\textcolor{red}{2467}], \mathbf{0}[\textcolor{red}{2468}]),\\
\hat{P}_5(3)&=\Conv(e_{3579\B}[\textcolor{red}{12357}],e_{579\B}[\textcolor{red}{12358}],e_{359\B}[\textcolor{red}{12367}],e_{59\B}[\textcolor{red}{12368}],e_{379\B}[\textcolor{red}{12457}],e_{79\B}[\textcolor{red}{12458}],e_{39\B}[\textcolor{red}{12467}],\\
&\ \ \ \ \ \ \ \ \ \
e_{9\B}[\textcolor{red}{12468}], e_{3789}[\textcolor{red}{13457}],e_{789}[\textcolor{red}{13458}],e_{389}[\textcolor{red}{13467}],e_{89}[\textcolor{red}{13468}],e_{358\A}[\textcolor{red}{13567}],
e_{58\A}[\textcolor{red}{13568}],\\
          &\ \ \ \ \ \ \ \ \ \ 
e_{578}[\textcolor{red}{13578}],e_{58}[\textcolor{red}{13678}],e_{38\A}[\textcolor{red}{14567}], e_{8\A}[\textcolor{red}{14568}],e_{78}[\textcolor{red}{14578}],e_{8}[\textcolor{red}{14678}],e_{379}[\textcolor{red}{23457}],e_{79}[\textcolor{red}{23458}],e_{39}[\textcolor{red}{23467}],\\
 &\ \ \ \ \ \ \ \ \ \ e_{9}[\textcolor{red}{23468}],e_{35\A}[\textcolor{red}{13568}],e_{5\A}[\textcolor{red}{23568}],e_{57}[\textcolor{red}{23578}],e_{5}[\textcolor{red}{23678}], e_{3\A}[\textcolor{red}{24567}],e_{\A}[\textcolor{red}{24568}],e_{7}[\textcolor{red}{24578}],\mathbf{0}[\textcolor{red}{24678}]),\\
\hat{P}_6(3)&=\Conv(e_{35\A\B}[\textcolor{red}{123567}],e_{5\A\B}[\textcolor{red}{123568}],e_{57\B}[\textcolor{red}{123578}],e_{5\B}[\textcolor{red}{123678}],e_{3\A\B}[\textcolor{red}{124567}],e_{\A\B}[\textcolor{red}{124568}],\\
 &\ \ \ \ \ \ \ \ \ \
e_{7\B}[\textcolor{red}{124578}],e_{\B}[\textcolor{red}{124678}],e_{38\A}[\textcolor{red}{134567}],e_{8\A}[\textcolor{red}{134568}],e_{78}[\textcolor{red}{134578}],\\
 &\ \ \ \ \ \ \ \ \ \
 e_{8}[\textcolor{red}{134678}],e_{3\A}[\textcolor{red}{234567}],e_{\A}[\textcolor{red}{234568}],e_{7}[\textcolor{red}{234578}],\mathbf{0}[\textcolor{red}{234678}]),\\
\hat{P}_7(3)&=\Conv(e_{3\A\CC}[\textcolor{red}{1234567}],e_{\A\CC}[\textcolor{red}{1234568}],e_{7\CC}[\textcolor{red}{1234578}],e_{\CC}[\textcolor{red}{1234678}],\\
 &\ \ \ \ \ \ \ \ \ \ 
e_{5\B}[\textcolor{red}{1235678}],e_{\B}[\textcolor{red}{1245678}],e_{8}[\textcolor{red}{1345678}],\mathbf{0}[\textcolor{red}{2345678}]).
    \end{aligned}
\end{equation*}

\end{example}

\bibliographystyle{plain}
\bibliography{bibliography.bib}

\end{document}